\documentclass[11pt,a4paper]{amsart}
\usepackage{amssymb,amsthm}

\usepackage[truedimen,margin=30truemm]{geometry}

\usepackage[pdftex,colorlinks=true,bookmarks=true,citecolor=blue,linkcolor=blue,pdfauthor={},pdftitle={}]{hyperref}

\theoremstyle{plain}
\newtheorem{theorem}{Theorem}[section]
\newtheorem{lemma}[theorem]{Lemma}

\theoremstyle{definition}

\newtheorem{assumption}[theorem]{Assumption}

\theoremstyle{remark}
\newtheorem{remark}[theorem]{Remark}

\numberwithin{equation}{section}

\begin{document}

\title[Two-jet Kolmogorov type flow on the unit sphere]{Linear stability and enhanced dissipation for the two-jet Kolmogorov type flow on the unit sphere}

\author[T.-H. Miura]{Tatsu-Hiko Miura}
\address{Department of Mathematics, Kyoto University, Kitashirakawa Oiwake-cho, Sakyo-ku, Kyoto 606-8502, Japan}
\email{t.miura@math.kyoto-u.ac.jp}



\begin{abstract}
  We consider the Navier--Stokes equations on the two-dimensional unit sphere and study the linear stability of the two-jet Kolmogorov type flow which is a stationary solution given by the zonal spherical harmonic function of degree two. We prove the linear stability of the two-jet Kolmogorov type flow for an arbitrary viscosity coefficient by showing the exponential decay of a solution to the linearized equation towards an equilibrium which grows as the viscosity coefficient tends to zero. The main result of this paper is the nonexistence of nonzero eigenvalues of the perturbation operator appearing in the linearized equation. By making use of the mixing property of the perturbation operator which is expressed by a recurrence relation for the spherical harmonics, we show that the perturbation operator does not have not only nonreal but also nonzero real eigenvalues. As an application of this result, we get the enhanced dissipation for the two-jet Kolmogorov type flow in the sense that a solution to the linearized equation rescaled in time decays arbitrarily fast as the viscosity coefficient tends to zero.
\end{abstract}

\maketitle

\section{Introduction} \label{S:Intro}

\subsection{Problem settings and main results} \label{SS:In_Pr}
Let $S^2$ be the unit sphere in $\mathbb{R}^3$. We consider the incompressible Navier--Stokes equations
\begin{align} \label{E:NS_Intro}
  \partial_t\mathbf{u}+\nabla_{\mathbf{u}}\mathbf{u}-\nu(\Delta_H\mathbf{u}+2\mathbf{u})+\nabla p = \mathbf{f}, \quad \mathrm{div}\,\mathbf{u} = 0 \quad\text{on}\quad S^2\times(0,\infty).
\end{align}
Here $\mathbf{u}$ is the velocity of a fluid, which is a tangential vector field on $S^2$, $p$ is the pressure, and $\mathbf{f}$ is a given external force. Also, $\nu>0$ is the viscosity coefficient, $\nabla_{\mathbf{u}}\mathbf{u}$ is the covariant derivative of $\mathbf{u}$ along itself, $\Delta_H$ is the Hodge Laplacian via identification of vector fields with one-forms, $\nabla$ is the gradient on $S^2$, and $\mathrm{div}$ is the divergence on $S^2$. Note that here we take the viscous term as twice of the divergence of the deformation tensor $\mathrm{Def}\,\mathbf{u}$:
\begin{align*}
  2\,\mathrm{div}\,\mathrm{Def}\,\mathbf{u} = \Delta_H\mathbf{u}+\nabla(\mathrm{div}\,\mathbf{u})+2\,\mathrm{Ric}(\mathbf{u}) = \Delta_H\mathbf{u}+\nabla(\mathrm{div}\,\mathbf{u})+2\mathbf{u},
\end{align*}
where $\mathrm{Ric}\equiv1$ is the Ricci curvature of $S^2$. We refer to \cite{EbiMar70,Aris89,DuMiMi06,Tay11_3,ChCzDi17} for this identity and the choice of the viscous term in the Navier--Stokes equations on manifolds.

Since $S^2$ is simply connected, the system \eqref{E:NS_Intro} is equivalent to the vorticity equation
\begin{align} \label{E:Vo_Intro}
  \partial_t\omega+\nabla_{\mathbf{u}}\omega-\nu(\Delta\omega+2\omega) = \mathrm{rot}\,\mathbf{f}, \quad \mathbf{u} = \mathbf{n}_{S^2}\times\nabla\Delta^{-1}\omega \quad\text{on}\quad S^2\times(0,\infty)
\end{align}
for the scalar vorticity $\omega=\mathrm{rot}\,\mathbf{u}$. Here $\nabla_{\mathbf{u}}\omega$ is the directional derivative of $\omega$ along $\mathbf{u}$ and $\Delta$ is the Laplace--Beltrami operator on $S^2$ which has the inverse $\Delta^{-1}$ in the space of $L^2$ functions on $S^2$ with zero mean. Also, $\mathbf{n}_{S^2}$ is the unit outward normal vector field of $S^2$ and $\times$ is the vector product in $\mathbb{R}^3$. We give derivation of \eqref{E:Vo_Intro} in Section \ref{S:DeVo} for the readers' convenience.

For $n\in\mathbb{N}$ and $a\in\mathbb{R}$ the vorticity equation \eqref{E:Vo_Intro} with external force $\mathrm{rot}\,\mathbf{f}_n^a=a\nu(\lambda_n-2)Y_n^0$ has a stationary solution with velocity field
\begin{align} \label{E:Zna_Intro}
  \omega_n^a(\theta,\varphi) = aY_n^0(\theta), \quad \mathbf{u}_n^a(\theta,\varphi) = -\frac{a}{\lambda_n\sin\theta}\frac{dY_n^0}{d\theta}(\theta)\partial_\varphi\mathbf{x}(\theta,\varphi).
\end{align}
Here $\theta$ and $\varphi$ are the colatitude and longitude so that $S^2$ is parametrized by
\begin{align} \label{E:SpCo_Intro}
  \mathbf{x}(\theta,\varphi) = (\sin\theta\cos\varphi,\sin\theta\sin\varphi,\cos\theta), \quad \theta\in[0,\pi], \, \varphi\in[0,2\pi).
\end{align}
Also, $\lambda_n=n(n+1)$ is the eigenvalue of $-\Delta$ and $Y_n^0$ is a corresponding eigenfunction (i.e. a spherical harmonic function) of longitudinal wavenumber zero given by
\begin{align*}
  Y_n^0(\theta) = \sqrt{\frac{2n+1}{4\pi}}P_n(\cos\theta), \quad P_n(s) = \frac{1}{2^nn!}\frac{d^n}{ds^n}(s^2-1)^n.
\end{align*}
The flow \eqref{E:Zna_Intro} is called the generalized Kolmogorov flow in \cite{Ili04} since it can be seen as a spherical version of the Kolmogorov flow in a two-dimensional (2D) flat torus. It is also called an $n$-jet zonal flow in \cite{SaTaYa13,SaTaYa15}. To emphasize both the similarity to the plane Kolmogorov flow and the number of jets, we call \eqref{E:Zna_Intro} the $n$-jet Kolmogorov type flow.

When $n=1$, it is shown in \cite{SaTaYa13} that the flow \eqref{E:Zna_Intro} is linearly stable for all $\nu>0$. In fact, the linearized equation for \eqref{E:Vo_Intro} around $\omega_1^a$ is of the form
\begin{align} \label{E:Li1_Eq}
  \partial_t\tilde{\omega}_1 = \nu(\Delta\tilde{\omega}_1+2\tilde{\omega}_1)-a_1\partial_\varphi(I+2\Delta^{-1})\tilde{\omega}_1, \quad a_1 = \frac{a}{4}\sqrt{\frac{3}{\pi}},
\end{align}
where $I$ is the identity operator (see Section \ref{S:DeVo} for derivation of the linearized equation), and the solution $\tilde{\omega}_1$ is expressed by the spherical harmonics $Y_n^m$ (see Section \ref{S:Pre}) as
\begin{align} \label{E:Li1_Sol}
  \begin{aligned}
    \tilde{\omega}_1(t) &= \sum_{n=1}^\infty\sum_{m=-n}^ne^{-\sigma_{n,m}t}(\tilde{\omega}_1(0),Y_n^m)_{L^2(S^2)}Y_n^m, \\
    \sigma_{n,m} &= \nu(\lambda_n-2)+ia_1m\left(1-\frac{2}{\lambda_n}\right).
  \end{aligned}
\end{align}
This formula gives the linear stability of $\omega_1^a$.

In this paper we are concerned with the linear stability of the two-jet Kolmogorov type flow. We substitute $\omega=\omega_2^a+\tilde{\omega}_2$ for \eqref{E:Vo_Intro} and omit the nonlinear term with respect to $\tilde{\omega}_2$ to get (see Section \ref{S:DeVo} for details)
\begin{align} \label{E:Li2_til}
  \partial_t\tilde{\omega}_2 = \nu(\Delta\tilde{\omega}_2+2\tilde{\omega}_2)-a_2\cos\theta\,\partial_\varphi(I+6\Delta^{-1})\tilde{\omega}_2, \quad a_2 = \frac{a}{4}\sqrt{\frac{5}{\pi}}.
\end{align}
Replacing $\tilde{\omega}_2$ and $a_2$ by $\omega$ and $a$, we rewrite \eqref{E:Li2_til} as
\begin{align} \label{E:Li2_Eq}
  \partial_t\omega = \mathcal{L}^{\nu,a}\omega = \nu A\omega-ia\Lambda\omega, \quad A = \Delta+2, \quad \Lambda = -i\cos\theta\,\partial_\varphi(I+6\Delta^{-1}).
\end{align}
We consider \eqref{E:Li2_Eq} in $L_0^2(S^2)$, the space of $L^2$ functions on $S^2$ with zero mean. Then $-A$ is a nonnegative self-adjoint operator on $L_0^2(S^2)$ and $\Lambda$ is an $A$-compact operator on $L_0^2(S^2)$, so $\mathcal{L}^{\nu,a}$ generates an analytic semigroup $\{e^{t\mathcal{L}^{\nu,a}}\}_{t\geq0}$ in $L_0^2(S^2)$ by a perturbation theory of semigroups (see \cite{EngNag00}). Thus the solution of \eqref{E:Li2_Eq} with initial data $\omega_0\in L_0^2(S^2)$ is given by $\omega(t)=e^{t\mathcal{L}^{\nu,a}}\omega_0$. We obtain the linear stability of the two-jet Kolmogorov type flow as follows.

\begin{theorem} \label{T:Li2_LS}
  There exists a constant $C>0$ such that
  \begin{align*}
    \|e^{t\mathcal{L}^{\nu,a}}\omega_0-\Omega_0^{\nu,a}\|_{L^2(S^2)} \leq C\left(1+\frac{|a|}{\nu}\right)e^{-4\nu t}\|\omega_0\|_{L^2(S^2)}
  \end{align*}
  for all $t\geq0$, $\nu>0$, $a\in\mathbb{R}$, and $\omega_0\in L_0^2(S^2)$, where
  \begin{align} \label{E:Li2S_Om}
    \Omega_0^{\nu,a} &= (\omega_0,Y_1^0)_{L^2(S^2)}Y_1^0+\sum_{m=\pm1}(\omega_0,Y_1^m)_{L^2(S^2)}\left(Y_1^m+\frac{a}{\nu}\frac{im}{2\sqrt{5}}Y_2^m\right).
  \end{align}
\end{theorem}

We present a more precise result in Theorem \ref{T:Li2S_Pre}. The proof is based on the observation that $\tilde{\omega}(t)=e^{t\mathcal{L}^{\nu,a}}\omega_0-\Omega_0^{\nu,a}$ also satisfies \eqref{E:Li2_Eq}. Using the expressions of $A$ and $\Lambda$ by $Y_n^m$ and noting that $(I+6\Delta^{-1})Y_2^m=0$, we take the $L^2(S^2)$-inner product of \eqref{E:Li2_Eq} for $\tilde{\omega}(t)$ with $Y_1^m$, $(I+6\Delta^{-1})\tilde{\omega}(t)$, and then $Y_2^m$ to get estimates for $\tilde{\omega}(t)$.

In Theorem \ref{T:Li2S_Pre} we also find that $e^{t\mathcal{L}^{\nu,a}}\omega_0-\Omega_0^{\nu,a}$ is orthogonal to $Y_1^0$ and $Y_1^{\pm1}$ in $L^2(S^2)$ for all $t\geq0$.
Moreover, $\Omega_0^{\nu,a}=0$ when $(\omega_0,Y_1^m)_{L^2(S^2)}=0$ for $m=0,\pm1$.
By these facts and Theorem \ref{T:Li2_LS}, we also have the following result shown in \cite{SaTaYa13}: the two-jet Kolmogorov type flow is globally asymptotically stable for all $\nu>0$ in the orthogonal complement of $\mathrm{span}\{Y_1^0,Y_1^{\pm1}\}$ in $L_0^2(S^2)$.

\begin{remark} \label{R:Li2_LS}
  The mapping $\omega_0\mapsto\Omega_0^{\nu,a}$ is a projection (but not the orthogonal one) from $L_0^2(S^2)$ onto the kernel of $\mathcal{L}^{\nu,a}$ in $L_0^2(S^2)$. We emphasize that the $Y_2^{\pm1}$-components remain in the equilibrium \eqref{E:Li2S_Om}, although $Y_2^{\pm1}$ themselves dissipate under the flow generated by $\mathcal{L}^{\nu,a}$. This is because the viscosity does not work for $Y_1^{\pm1}$, i.e. $AY_1^{\pm1}=0$ and $\Lambda Y_1^{\pm1}$ have a contribution to the $Y_2^{\pm1}$-components. Also, if $\omega_0$ is real-valued, then
  \begin{align*}
    \Omega_0^{\nu,a} = (\omega_0,Y_1^0)_{L^2(S^2)}Y_1^0+2\mathrm{Re}\Bigl((\omega_0,Y_1^1)_{L^2(S^2)}Y_1^1\Bigr)-\frac{a}{\nu}\frac{1}{\sqrt{5}}\mathrm{Im}\Bigl((\omega_0,Y_1^1)_{L^2(S^2)}Y_2^1\Bigr)
  \end{align*}
  by $Y_n^{-m}=(-1)^m\overline{Y_n^m}$ (see Section \ref{S:Pre}). Hence $\Omega_0^{\nu,a}$ may grow as $\nu\to0$ even if we only consider real-valued solutions to \eqref{E:Li2_Eq}.
\end{remark}

As mentioned above, when $(\omega_0,Y_1^m)_{L^2(S^2)}=0$ for $m=0,\pm1$, we have $\Omega_0^{\nu,a}=0$ and thus $e^{t\mathcal{L}^{\nu,a}}\omega_0$ decays at the rate $O(e^{-\nu t})$ by Theorem \ref{T:Li2_LS}.
In the case of the plane Kolmogorov flow \cite{BecWay13,LinXu19,IbMaMa19,WeiZha19,WeZhZh20}, however, it is shown that a solution to the linearized equation decays at a rate faster than $O(e^{-\nu t})$ when $\nu$ is sufficiently small.
Such a phenomenon is called the enhanced dissipation and also observed in the study of an advection-diffusion equation \cite{CoKiRyZl08,Zla10,Wei21}.
Our next aim is to study the enhanced dissipation for the solution $e^{t\mathcal{L}^{\nu,a}}\omega_0$ to \eqref{E:Li2_Eq} when $\omega_0$ is orthogonal to $Y_1^0$ and $Y_1^{\pm1}$ in $L^2(S^2)$.

To carry out a more detailed analysis of the behavior of $e^{t\mathcal{L}^{\nu,a}}\omega_0$ as $\nu\to0$ (or $t\to\infty$), it is important to study the spectral properties of $\Lambda$. In the study of the linear stability of the Kolmogorov type flows for the Euler equations on $S^2$, Taylor \cite{Tay16} showed that the spectrum of $\Lambda$ lies on the real axis and in particular $\Lambda$ does not have nonreal eigenvalues. The next theorem is the main result of this paper in which we further show that $\Lambda$ does not have real eigenvalues except for zero.

\begin{theorem} \label{T:NoEi_Lam}
  The operator $\Lambda$ in $L_0^2(S^2)$ has no eigenvalues in $\mathbb{C}\setminus\{0\}$.
\end{theorem}

We prove Theorem \ref{T:NoEi_Lam} in Section \ref{S:No_Ei}. For the proof, we take an approach making use of the \textit{mixing} structure of $\Lambda$, which seems to have its own interest as the spectral analysis of linear operators; the key idea is outlined in Section \ref{SS:In_Ou} below. A common approach to the proof of the nonexistence of eigenvalues of a linear operator is the one based on the theory of ordinary differential equations (ODEs): one analyzes an ODE associated with the linear operator and applies the uniqueness of a (smooth) solution to the ODE to show that a solution to the eigenvalue problem identically vanishes. Such an ODE approach is used in the analysis of the Burgers vortex \cite{Mae11} and of the plane Kolmogorov flow \cite{LinXu19,IbMaMa19,WeZhZh20}, but it does not work efficiently in our case because of the size of the coefficient of the nonlocal operator $\Delta^{-1}$ in $\Lambda$, which is the crucial difficulty in the proof of Theorem \ref{T:NoEi_Lam}. Instead, to overcome this difficulty, we employ the mixing property of $\Lambda$ expressed by a recurrence relation for the spherical harmonics, which are the basis functions of $L^2(S^2)$. The main novelty of this paper is to give a new and robust approach to analyze the eigenvalue problem for a linear operator with a suitable mixing property.

As an application of Theorem \ref{T:NoEi_Lam}, we show that the enhanced dissipation occurs for the rescaled flow $e^{\frac{t}{\nu}\mathcal{L}^{\nu,a}}\omega_0$, which is a solution to $\partial_t\omega=A\omega-i\alpha\Lambda\omega$ with $\alpha=a/\nu$, as in the case of an advection-diffusion equation \cite{CoKiRyZl08,Zla10,Wei21}.
Let
\begin{align*}
  \mathcal{X} = \{u\in L_0^2(S^2) \mid (u,Y_n^0)_{L^2(S^2)}=(u,Y_1^m)_{L^2(S^2)}=0, \, n\geq1, \, |m|=0,1\},
\end{align*}
which is a closed subspace of $L_0^2(S^2)$ invariant under the actions of $A$ and $\Lambda$ (see Section \ref{S:EnDi}). By $\mathbb{Q}$ we denote the orthogonal projection from $\mathcal{X}$ onto the orthogonal complement of the kernel of $\Lambda$ restricted on $\mathcal{X}$. We have the enhanced dissipation for the rescaled flow in the following sense.

\begin{theorem} \label{T:EnDi_Lna}
  For each $\tau>0$ we have
  \begin{align} \label{E:EnDi_Lna}
    \lim_{|a/\nu|\to\infty}\sup_{t\geq\tau}\|\mathbb{Q}e^{\frac{t}{\nu}\mathcal{L}^{\nu,a}}\|_{\mathcal{X}\to\mathcal{X}} = 0.
  \end{align}
\end{theorem}

We establish Theorem \ref{T:EnDi_Lna} in Section \ref{S:EnDi} by using abstract results for a perturbed operator given in Section \ref{S:Abst}. In particular, we prove a decay estimate for the semigroup generated by an $m$-accretive operator on a weighted Hilbert space which is a version of the Gearhart--Pr\"{u}ss type theorem shown by Wei \cite{Wei21} and combine it with a convergence result for the pseudospectral bound given by Ibrahim, Maekawa, and Masmoudi \cite{IbMaMa19}.

Theorem \ref{T:EnDi_Lna} implies that the rescaled flow $\mathbb{Q}e^{\frac{t}{\nu}\mathcal{L}^{\nu,a}}\omega_0$ converges to zero in $L^2(S^2)$ as $\nu\to0$ for each fixed $t>0$ and $a\in\mathbb{R}$, but does not give the actual convergence rate. The original flow $\mathbb{Q}e^{t\mathcal{L}^{\nu,a}}\omega_0$ is expected to decay at the rate $O(e^{-\sqrt{\nu}\,t})$ as in the case of the plane Kolmogorov flow \cite{BecWay13,IbMaMa19,WeiZha19,WeZhZh20}, and this will be rigorosuly verified in the companion paper \cite{MaeMiupre}. It is stressed here, however, that Theorem \ref{T:NoEi_Lam} requires the most original idea in the sphere case.

\begin{remark} \label{R:Li1_Sol}
  The enhanced dissipation of the form \eqref{E:EnDi_Lna} does not occur for a perturbation of the one-jet Kolmogorov type flow since the solution $\tilde{\omega}_1$ to \eqref{E:Li1_Eq} is determined by \eqref{E:Li1_Sol}. This seems to be natural from the point of view that the velocity $\mathbf{u}_1^a$ given by \eqref{E:Zna_Intro} with $n=1$ is a rigid rotation around the $x_3$-axis, i.e. a Killing vector field which does not have a mixing effect in the sense that it generates a one-parameter group of isometries of $S^2$.
\end{remark}

\subsection{Outline of the proof of Theorem \ref{T:NoEi_Lam}} \label{SS:In_Ou}
To prove Theorem \ref{T:NoEi_Lam}, we make use of the mixing structure of $\Lambda$ expressed by the recurrence relation
\begin{align} \label{E:Cos_Intro}
  \cos\theta \, Y_n^m = a_n^mY_{n-1}^m+a_{n+1}^mY_{n+1}^m
\end{align}
with nonzero coefficients $a_n^m$ (see \eqref{E:Y_Rec}). By the Fourier series expansion of $\Lambda$ with respect to the longitude $\varphi$, it is sufficient to show that for each $m\in\mathbb{Z}\setminus\{0\}$ and $\mu\in\mathbb{C}\setminus\{0\}$ the equation
\begin{align} \label{E:Eim_Intro}
  \mu u = \cos\theta \, Bu, \quad u \in \mathcal{P}_mL_0^2(S^2) = \overline{\mathrm{span}}\{Y_n^m \mid n\geq |m|\}
\end{align}
admits only a trivial solution $u=0$, where $B=I+6\Delta^{-1}$. Applying \eqref{E:Cos_Intro} and
\begin{align} \label{E:B_Intro}
  BY_n^m = \left(1-\frac{6}{\lambda_n}\right)Y_n^m
\end{align}
to \eqref{E:Eim_Intro}, we have $(u,Y_1^m)_{L^2(S^2)}=0$ when $|m|=1$. Then we easily get $u=0$ if $\mathrm{Im}\,\mu\neq0$ by taking the imaginary part of the $L^2(S^2)$-inner product of \eqref{E:Eim_Intro} with $Bu$ and using \eqref{E:B_Intro}. Also, when $\mu\in\mathbb{R}$ and $|\mu|\geq1$, we apply $|\cos\theta|\leq1$ and \eqref{E:B_Intro} to the $L^2(S^2)$-norm of \eqref{E:Eim_Intro} to find that $u=0$. The most difficult case is $\mu\in\mathbb{R}$ and $|\mu|<1$. In this case, one may try to use the ODE approach as in the flat torus case \cite{LinXu19,IbMaMa19,WeZhZh20}: since $v=\Delta^{-1}u\in\mathcal{P}_mL_0^2(S^2)$ is of the form $v=V(\theta)e^{im\varphi}$, one can rewrite \eqref{E:Eim_Intro} as a second order ODE for $V$ with singularity of order one at $\theta=\theta_\mu=\arccos\mu$. Then one can apply the uniqueness of a $C^1$ solution of the ODE to get $V\equiv0$ if one shows that $V$ is of class $C^1$ and vanishes at $\theta=\theta_\mu$ along with its derivative $V'$. In our case, however, it seems to be too difficult to show $V'(\theta_\mu)=0$ when $|m|=1,2$ since the coefficient $6$ of $\Delta^{-1}$ in $B$ is too large compared to $m^2=1,4$ appearing in the expression of $\Delta v$ for $v=V(\theta)e^{im\varphi}$ under the spherical coordinate system. To overcome this difficulty, we apply the mixing property \eqref{E:Cos_Intro} instead of using the ODE approach. Indeed, using \eqref{E:B_Intro} and the expression of $Y_{n+1}^m$ in terms of $\cos\theta\,Y_n^m$ and $Y_{n-1}^m$ by \eqref{E:Cos_Intro}, we rewrite \eqref{E:Eim_Intro} as
\begin{align} \label{E:UN_Intro}
  (\mu-x_3)(w_{<N}+Bu_{\geq N}) = \sigma_N^mY_{|m|}^m+6\mu\Delta^{-1}u_{\geq N} \quad (x_3=\cos\theta),
\end{align}
where $u_{\geq N}=\sum_{n\geq N}(u,Y_n^m)_{L^2(S^2)}Y_n^m$ with a large $N\in\mathbb{N}$, $w_{<N}\in\mathrm{span}\{Y_{|m|}^m,\dots,Y_{N-1}^m\}$, and $\sigma_N^m\in\mathbb{C}$. Then, setting $x_3=\mu$ in \eqref{E:UN_Intro} and noting that $Y_{|m|}^m\neq0$ at $x_3=\mu$, we find that $|\sigma_N^m|$ is bounded by $\|\Delta^{-1}u_{\geq N}\|_{L^\infty(S^2)}$ and thus by $\|(-\Delta)^{-1/2}u_{\geq N}\|_{L^2(S^2)}$ (see Lemma \ref{L:Linf}). We take the $L^2(S^2)$-inner product of \eqref{E:UN_Intro} divided by $\mu-x_3$ with $Bu_{\geq N}$. Then we use the Hardy type inequality for $\sigma_N^mY_{|m|}^m+6\mu\Delta^{-1}u_{\geq N}$ (see Lemma \ref{L:Hardy}), the estimate for $|\sigma_N^m|$, and $(Bu_{\geq N},w_{<N})_{L^2(S^2)}=0$ to get $\|Bu_{\geq N}\|_{L^2(S^2)}\leq C_{m,\mu}\|(-\Delta)^{-1/2}u_{\geq N}\|_{L^2(S^2)}$ with a constant $C_{m,\mu}$ depending only on $m$ and $\mu$. In this inequality, the left-hand side is bounded below by $\|u_{\geq N}\|_{L^2(S^2)}/2$ by \eqref{E:B_Intro}, while the right-hand side is bounded above by $C_{m,\mu}\lambda_N^{-1}\|u_{\geq N}\|_{L^2(S^2)}$ with $\lambda_N=N(N+1)$ by $(-\Delta)^{-1/2}Y_n^m=\lambda_n^{-1}Y_n^m$. Hence $u_{\geq N}=0$ for a sufficiently large $N$. Then we again use the mixing property \eqref{E:Cos_Intro} and $u_{\geq N}=0$ to \eqref{E:Eim_Intro} to get $(u,Y_n^m)_{L^2(S^2)}=0$ for $n=N-1,\dots,|m|$ inductively, and thus $u=0$.

The proof outlined here relies on the properties of the basis functions $Y_n^m$ of $L^2(S^2)$. In particular, the recurrence relation \eqref{E:Cos_Intro} representing the mixing structure of $\Lambda$ is crucial for the proof. Our approach may be applicable to the spectral analysis of other linear operators, especially in the higher dimensional case where it seems to be difficult to apply the ODE approach.

\subsection{Literature overview} \label{SS:In_Li}
The Navier--Stokes equations on spheres and more general manifolds appear in various fields such as geophysical fluid dynamics and biology. Several authors studied the Navier--Stokes and vorticity equations on spheres and manifolds \cite{IliFil88,Ili90,CaRaTi99,Ili04,Wir15,Lic16,Ski17}, but in these works the viscous term is taken to be $\nu\Delta_H\mathbf{u}$ (without curvature term) by analogy of the flat domain case. There are also a lot of works \cite{Tay92,Prie94,Nag99,MitTay01,DinMit04,KheMis12,ChaCzu13,ChaYon13,SaTaYa13,SaTaYa15,Pie17,KohWen18,PrSiWi20,SamTuo20} on the Navier--Stokes equations on manifolds in which the viscous term has the curvature term $\mathrm{Ric}(\mathbf{u})$ as in \eqref{E:NS_Intro}.

The linear stability of the Kolmogorov type flows for the Navier--Stokes equations on a sphere was studied by Ilyin \cite{Ili04} with viscous term $\nu\Delta_H\mathbf{u}$ and by Sasaki, Takehiro, and Yamada \cite{SaTaYa13} with viscous term $\nu(\Delta_H\mathbf{u}+2\mathbf{u})$ as in \eqref{E:NS_Intro}. Sasaki, Takehiro, and Yamada \cite{SaTaYa15} also studied the nonlinear stability of the Kolmogorov type flows. They showed that the $n$-jet Kolmogorov type flow is globally stable for all $\nu>0$ when $n=1,2$ and unstable for a small $\nu>0$ when $n\geq3$. We note that Theorem \ref{T:Li2_LS} is not covered by the results of the above papers since the viscous term is different in \cite{Ili04} and the stability of the two-jet Kolmogorov type flow is studied in a slightly smaller space in \cite{SaTaYa13,SaTaYa15}. In particular, we find that a solution to \eqref{E:Li2_Eq} may converge to a nonzero function as $t\to\infty$ in the whole $L_0^2(S^2)$, while it is shown in \cite{SaTaYa13,SaTaYa15} that a perturbation of the two-jet Kolmogorov type flow just converges to zero as $t\to\infty$ in the orthogonal complement of $\mathrm{span}\{Y_1^0,Y_1^{\pm1}\}$ in $L_0^2(S^2)$.
Also, Taylor \cite{Tay16} studied the linear stability of the Kolmogorov type flows for the Euler equations on a sphere.

Let us also mention the Kolmogorov flow in a 2D flat torus. The Kolmogorov flow is a stationary solution to the 2D Navier--Stokes equations in a flat torus with shear external force. By Iudovich \cite{Iud65} it was shown that the Kolmogorov flow in the square torus is globally stable for an arbitrary viscosity coefficient (see also \cite{Mar86}). It is also known that the Kolmogorov flow may become unstable when the length of the periodicity in one direction is changed (see e.g. \cite{MesSin61,Iud65,OkaSho93,MatMiy02}). In the stable case, Beck and Wayne \cite{BecWay13} numerically conjectured that a perturbation of the Kolmogorov flow rapidly decays at the rate $O(e^{-\sqrt{\nu}\,t})$ compared to the usual one $O(e^{-\nu t})$ when the viscosity coefficient $\nu$ is sufficiently small. They also verified this enhanced dissipation for a linearized operator without a nonlocal term by using the hypocoercivity method developed by Villani \cite{Vil09}. Lin and Xu \cite{LinXu19} studied the full linearized operator and also the nonlinear problem. They proved the enhanced dissipation in both cases but without an explicit decay rate based on the Hamiltonian structure of a perturbation operator and the RAGE theorem which is used in the study of the enhanced dissipation for an advection-diffusion equation \cite{CoKiRyZl08,Zla10}. The enhanced dissipation for the linearized problem with the decay rate $O(e^{-\sqrt{\nu}\,t})$ was confirmed by Ibrahim, Maekawa, and Masmoudi \cite{IbMaMa19} based on the pseudospectral bound method, by Wei and Zhang \cite{WeiZha19} based on the hypocoercivity method, and by Wei, Zhang, and Zhao \cite{WeZhZh20} based on the wave operator method.

\subsection{Organization of this paper} \label{SS:In_Or}
The rest of this paper is organized as follows. Section \ref{S:Pre} gives basic facts of calculus on $S^2$. In Section \ref{S:LiSt} we study the linear stability of the two-jet Kolmogorov type flow. Section \ref{S:No_Ei} is devoted to the proof of Theorem \ref{T:NoEi_Lam}. In Section \ref{S:EnDi} we show that the enhanced dissipation occurs for the rescaled flow $e^{\frac{t}{\nu}\mathcal{L}^{\nu,a}}\omega_0$. Also, in Section \ref{S:Abst} we give abstract results used in the study of the enhanced dissipation. In Section \ref{S:DeVo} we derive the vorticity equation \eqref{E:Vo_Intro} and linearize it around the $n$-jet Kolmogorov type flow.

\section{Preliminaries} \label{S:Pre}
In this section we give basic facts of calculus on $S^2$.

Let $S^2$ be the unit sphere in $\mathbb{R}^3$ equipped with the Riemannian metric induced by the Euclidean metric of $\mathbb{R}^3$. We denote by $\theta$ and $\varphi$ the colatitude and longitude so that $S^2$ is parametrized by \eqref{E:SpCo_Intro}. For a (complex-valued) function $u$ on $S^2$, we sometimes abuse the notation
\begin{align*}
  u(\theta,\varphi) = u(\sin\theta\cos\varphi,\sin\theta\sin\varphi,\cos\theta), \quad \theta\in[0,\pi],\,\varphi\in[0,2\pi)
\end{align*}
when no confusion may occur. Thus the gradient of $u$ is expressed as
\begin{align} \label{E:Grad_S2}
  \nabla u = \partial_\theta u
  \begin{pmatrix}
    \cos\theta\cos\varphi \\
    \cos\theta\sin\varphi \\
    -\sin\theta
  \end{pmatrix}
  +\frac{\partial_\varphi u}{\sin^2\theta}
  \begin{pmatrix}
    -\sin\theta\sin\varphi \\
    \sin\theta\cos\varphi \\
    0
  \end{pmatrix}
\end{align}
and the integral of $u$ over $S^2$ is given by
\begin{align} \label{E:Int_S2}
  \int_{S^2}u\,d\mathcal{H}^2 = \int_0^{2\pi}\left(\int_0^\pi u(\theta,\varphi)\sin\theta\,d\theta\right)\,d\varphi,
\end{align}
where $\mathcal{H}^k$ is the Hausdorff measure of dimension $k\in\mathbb{N}$. As usual, we set
\begin{align*}
  (u,v)_{L^2(S^2)} = \int_{S^2}u\bar{v}\,d\mathcal{H}^2, \quad \|u\|_{L^2(S^2)} = (u,u)_{L^2(S^2)}^{1/2}, \quad u, v\in L^2(S^2),
\end{align*}
where $\bar{v}$ is the complex conjugate of $v$, and write $H^k(S^2)$, $k\in\mathbb{Z}_{\geq0}$ for the Sobolev spaces of $L^2$ functions on $S^2$ with $H^0(S^2)=L^2(S^2)$ (see \cite{Aub98}).

Let $\Delta$ be the Laplace--Beltrami operator on $S^2$. It is well known (see e.g. \cite{VaMoKh88,Tes14}) that $\lambda_n=n(n+1)$ is an eigenvalue of $-\Delta$ with multiplicity $2n+1$ for each $n\in\mathbb{Z}_{\geq0}$ and the corresponding eigenfunctions are the spherical harmonics
\begin{align} \label{E:SpHa}
  Y_n^m = Y_n^m(\theta,\varphi) = \sqrt{\frac{2n+1}{4\pi}\frac{(n-m)!}{(n+m)!}} \, P_n^m(\cos\theta)e^{im\varphi}, \quad m=0,\pm1,\dots,\pm n.
\end{align}
Here $P_n^0$, $n\in\mathbb{Z}_{\geq0}$ are the Legendre polynomials defined as
\begin{align*}
  P_n^0(s) = P_n(s) = \frac{1}{2^nn!}\frac{d^n}{ds^n}(s^2-1)^n, \quad s\in(-1,1)
\end{align*}
and the associated Legendre functions $P_n^m$, $n\in\mathbb{Z}_{\geq0}$, $|m|\leq n$ are given by
\begin{align*}
  P_n^m(s) =
  \begin{cases}
    (-1)^m(1-s^2)^{m/2}\displaystyle\frac{d^m}{ds^m}P_n(s), &m\geq0, \\
    (-1)^{|m|}\displaystyle\frac{(n-|m|)!}{(n+|m|)!}P_n^{|m|}(s), &m = -|m| < 0
  \end{cases}
\end{align*}
so that $Y_n^{-m}=(-1)^m\overline{Y_n^m}$ (see \cite{Led72,DLMF}). Moreover, the set of all $Y_n^m$ forms an orthonormal basis of $L^2(S^2)$, i.e. for each $u\in L^2(S^2)$ we have $u=\sum_{n=0}^\infty\sum_{m=-n}^nc_n^mY_n^m$ with $c_n^m=(u,Y_n^m)_{L^2(S^2)}$. Note that, here and in what follows, the superscript $m$ of coefficients just corresponds to that of $Y_n^m$ and does not mean the $m$-th power unless otherwise stated. It is also known that the recurrence relation
\begin{align*}
  (n-m+1)P_{n+1}^m(s)-(2n+1)sP_n^m(s)+(n+m)P_{n-1}^m(s) = 0
\end{align*}
holds (see \cite[(7.12.12)]{Led72}) and thus (see also \cite[Section 5.7]{VaMoKh88})
\begin{align} \label{E:Y_Rec}
  \cos\theta\,Y_n^m = a_n^mY_{n-1}^m+a_{n+1}^mY_{n+1}^m, \quad a_n^m = \sqrt{\frac{(n-m)(n+m)}{(2n-1)(2n+1)}}
\end{align}
for $n\in\mathbb{Z}_{\geq0}$ and $|m|\leq n$, where we consider $Y_{|m|-1}^m\equiv0$.

Let $L_0^2(S^2)$ be the space of $L^2$ functions on $S^2$ with zero mean, i.e.
\begin{align*}
  L_0^2(S^2) = \left\{u\in L^2(S^2) ~\middle|~ \int_{S^2}u\,d\mathcal{H}^2 = 0\right\} = \{u\in L^2(S^2) \mid (u,Y_0^0)_{L^2(S^2)}=0\}.
\end{align*}
Then $\Delta$ is invertible, self-adjoint, and with compact resolvent as a linear operator
\begin{align*}
  \Delta\colon D_{L_0^2(S^2)}(\Delta) \subset L_0^2(S^2) \to L_0^2(S^2), \quad D_{L_0^2(S^2)}(\Delta) = L_0^2(S^2)\cap H^2(S^2).
\end{align*}
Also, for $s\in\mathbb{R}$, the operator $(-\Delta)^s$ is defined on $L_0^2(S^2)$ by
\begin{align} \label{E:Def_Laps}
  (-\Delta)^su = \sum_{n=1}^\infty\sum_{m=-n}^n\lambda_n^s(u,Y_n^m)_{L^2(S^2)}Y_n^m, \quad u \in L_0^2(S^2).
\end{align}
We easily observe by a density argument and integration by parts that
\begin{align} \label{E:Lap_Half}
  \|(-\Delta)^{1/2}u\|_{L^2(S^2)} = \|\nabla u\|_{L^2(S^2)}, \quad u\in L_0^2(S^2)\cap H^1(S^2).
\end{align}
Let $u$ be a function on $S^2$. We write $u=U(\theta)e^{im\varphi}$ if $u$ is of the form
\begin{align*}
  u(\sin\theta\cos\varphi,\sin\theta\sin\varphi,\cos\theta) = U(\theta)e^{im\varphi}, \quad \theta\in[0,\pi], \, \varphi\in[0,2\pi)
\end{align*}
with some function $U$ of the colatitude $\theta$ and $m\in\mathbb{Z}$. In this case, we have
\begin{align} \label{E:L2_Mode}
  \begin{aligned}
    \|u\|_{L^2(S^2)}^2 &= 2\pi\int_0^\pi|U(\theta)|^2\sin\theta\,d\theta, \\
    \|\nabla u\|_{L^2(S^2)}^2 &= 2\pi\int_0^\pi\left(|U'(\theta)|^2+\frac{m^2}{\sin^2\theta}|U(\theta)|^2\right)\sin\theta\,d\theta, \quad U' = \frac{dU}{d\theta}.
  \end{aligned}
\end{align}
When $u=U(\theta)e^{im\varphi}$ is in $L^2(S^2)$, we can write $u=\sum_{n\geq|m|}c_n^mY_n^m$ since $(u,Y_n^{m'})_{L^2(S^2)}=0$ for $m'\neq m$. In particular, if $m\neq0$, then $u\in L_0^2(S^2)$ and we can use \eqref{E:Lap_Half} to $u$.

For a function $u$ on $S^2$ and $m\in\mathbb{Z}$ we define a function $\mathcal{P}_mu$ on $S^2$ by
\begin{multline} \label{E:Def_Proj}
  \mathcal{P}_mu(\sin\theta\cos\varphi,\sin\theta\sin\varphi,\cos\theta) \\
  = \frac{e^{im\varphi}}{2\pi}\int_0^{2\pi}u(\sin\theta\cos\phi,\sin\theta\sin\phi,\cos\theta)e^{-im\phi}\,d\phi.
\end{multline}
Note that $\mathcal{P}_mu=u$ and $\mathcal{P}_{m'}u=0$ for $m'\neq m$ if $u=U(\theta)e^{im\varphi}$.

\begin{lemma} \label{L:NS_Pole}
  If $u=U(\theta)e^{im\varphi}\in C(S^2)$ with $m\in\mathbb{Z}\setminus\{0\}$, then $U(0)=U(\pi)=0$.
\end{lemma}

\begin{proof}
  Since $u=\mathcal{P}_mu$, we set $\theta=0,\pi$ and $\varphi=0$ in \eqref{E:Def_Proj} to get
  \begin{align*}
    u(0,0,\pm1) = \mathcal{P}_mu(0,0,\pm1) = \frac{1}{2\pi}\int_0^{2\pi}u(0,0,\pm1)e^{-im\phi}\,d\phi = 0,
  \end{align*}
  where the last equality follows from $m\neq0$. Hence $U(0)=U(\pi)=0$.
\end{proof}

\begin{lemma} \label{L:Proj_Bdd}
  For $\theta_1,\theta_2\in[0,\pi]$ with $\theta_1\leq\theta_2$ let
  \begin{align} \label{E:Def_Band}
    S^2(\theta_1,\theta_2)=\{(x_1,x_2,x_3)\in S^2 \mid x_3 = \cos\theta,\,\theta\in(\theta_1,\theta_2) \}.
  \end{align}
  Then for each $m\in\mathbb{Z}$ we have
  \begin{align} \label{E:Proj_Bdd}
    \begin{alignedat}{2}
      \|\mathcal{P}_mu\|_{L^2(S^2(\theta_1,\theta_2))} &\leq \|u\|_{L^2(S^2(\theta_1,\theta_2))}, &\quad &u\in L^2(S^2(\theta_1,\theta_2)), \\
      \|\nabla\mathcal{P}_mu\|_{L^2(S^2(\theta_1,\theta_2))} &\leq \|\nabla u\|_{L^2(S^2(\theta_1,\theta_2))}, &\quad &u\in H^1(S^2(\theta_1,\theta_2)).
    \end{alignedat}
  \end{align}
\end{lemma}

\begin{proof}
  The first inequality of \eqref{E:Proj_Bdd} follows from \eqref{E:Int_S2} and
  \begin{align} \label{Pf_PB:Hol}
    |\mathcal{P}_mu(\theta,\varphi)| \leq \frac{1}{2\pi}\int_0^{2\pi}|u(\theta,\phi)|\,d\phi \leq \frac{1}{\sqrt{2\pi}}\left(\int_0^{2\pi}|u(\theta,\phi)|^2\,d\phi\right)^{1/2}
  \end{align}
  by H\"{o}lder's inequality. Also, since $\partial_\theta\mathcal{P}_mu=\mathcal{P}_m(\partial_\theta u)$ and
  \begin{align*}
    \mathcal{P}_m(\partial_\varphi u)(\theta,\varphi) &= \frac{e^{im\varphi}}{2\pi}\int_0^{2\pi}\partial_\phi u(\theta,\phi)e^{-im\varphi}\,d\phi = \frac{ime^{im\varphi}}{2\pi}\int_0^{2\pi}u(\theta,\phi)e^{-im\phi}\,d\phi \\
    &= \partial_\varphi\mathcal{P}_mu(\theta,\varphi)
  \end{align*}
  by integration by parts and $u(\theta,0)=u(\theta,2\pi)$, we have the second inequality of \eqref{E:Proj_Bdd} by \eqref{E:Grad_S2}, \eqref{E:Int_S2}, and \eqref{Pf_PB:Hol} with $u$ replaced by $\partial_\theta u$ and $\partial_\varphi u$.
\end{proof}

\begin{lemma} \label{L:Dense}
  For $m\in\mathbb{Z}$ let $u=U(\theta)e^{im\varphi}\in H^1(S^2)$. Then there exist smooth functions on $S^2$ of the form $u_k=U_k(\theta)e^{im\varphi}$, $k\in\mathbb{N}$ that converge to $u$ strongly in $H^1(S^2)$.
\end{lemma}

\begin{proof}
  Since $S^2$ is compact and without boundary, we can take smooth functions $v_k$, $k\in\mathbb{N}$ that converge to $u$ strongly in $H^1(S^2)$ by standard localization and mollification arguments. Then $u_k=\mathcal{P}_mv_k$, $k\in\mathbb{N}$ are smooth functions of the form $u_k=U_k(\theta)e^{im\varphi}$ by \eqref{E:Def_Proj}. Moreover, since $\mathcal{P}_mu=u$ by $u=U(\theta)e^{im\varphi}$, we have
  \begin{align*}
    \|u-u_k\|_{H^1(S^2)} = \|\mathcal{P}_mu-\mathcal{P}_mv_k\|_{H^1(S^2)} \leq \|u-v_k\|_{H^1(S^2)} \to 0 \quad\text{as}\quad k\to\infty
  \end{align*}
  by \eqref{E:Proj_Bdd} and the strong convergence of $\{v_k\}_{k=1}^\infty$ to $u$ in $H^1(S^2)$.
\end{proof}

\begin{lemma} \label{L:Linf}
  For $m\in\mathbb{Z}\setminus\{0\}$ let $u=U(\theta)e^{im\varphi}\in H^1(S^2)$. Then
  \begin{align} \label{E:Linf}
    \|u\|_{L^\infty(S^2)}^2 = \|U\|_{L^\infty(0,\pi)}^2 \leq \frac{1}{\pi|m|}\|(-\Delta)^{1/2}u\|_{L^2(S^2)}^2.
  \end{align}
\end{lemma}

\begin{proof}
  By Lemma \ref{L:Dense} we may assume $u=U(\theta)e^{im\varphi}\in C^\infty(S^2)$. Then $U(0)=0$ by Lemma \ref{L:NS_Pole} since $m\neq0$. Hence for $\theta\in(0,\pi)$ we have
  \begin{align*}
    |U(\theta)|^2 &= \int_0^{\theta}\frac{d}{d\vartheta}|U(\vartheta)|^2\,d\vartheta \leq 2\int_0^\pi|U(\vartheta)||U'(\vartheta)|\,d\vartheta \\
    &\leq 2\left(\int_0^\pi\frac{|U(\vartheta)|^2}{\sin\vartheta}\,d\vartheta\right)^{1/2}\left(\int_0^\pi|U'(\vartheta)|^2\sin\vartheta\,d\vartheta\right)^{1/2} \leq \frac{1}{\pi|m|}\|\nabla u\|_{L^2(S^2)}^2
  \end{align*}
  by H\"{o}lder's inequality and \eqref{E:L2_Mode}. By this inequality and \eqref{E:Lap_Half} we get \eqref{E:Linf}.
\end{proof}

\begin{lemma} \label{L:Hardy}
  Let $\mu\in(-1,1)$ and $\theta_\mu=\arccos\mu\in(0,\pi)$. Then
  \begin{align} \label{E:Hardy}
    \int_{\theta_1}^{\theta_2}\left|\frac{U(\theta)\sqrt{\sin\theta}-U(\theta_\mu)\sqrt{\sin\theta_\mu}}{\mu-\cos\theta}\right|^2\,d\theta \leq \frac{16}{\pi\sin^2\theta_\mu}\|\nabla u\|_{L^2(S^2(\theta_1,\theta_2))}^2
  \end{align}
  for all $\theta_1\in[0,\theta_\mu]$, $\theta_2\in[\theta_\mu,\pi]$, and $u=U(\theta)e^{im\varphi}\in H^1(S^2(\theta_1,\theta_2))$ with $m\in\mathbb{Z}\setminus\{0\}$, where $S^2(\theta_1,\theta_2)$ is given by \eqref{E:Def_Band}.
\end{lemma}

\begin{proof}
  For $\theta\in[0,\pi]$ let $\theta_{\max}=\max\{\theta,\theta_\mu\}$ and $\theta_{\min}=\min\{\theta,\theta_\mu\}$.
  Since
  \begin{align*}
    |\mu-\cos\theta| = \int_{\theta_{\min}}^{\theta_{\max}}\sin\vartheta\,d\vartheta = (\theta_{\max}-\theta_{\min})\int_0^1\sin\bigl((1-t)\theta_{\min}+t\theta_{\max}\bigr)\,dt
  \end{align*}
  by $\mu=\cos\theta_\mu$, and since $\sin\theta$ is concave for $\theta\in[0,\pi]$, we easily find that
  \begin{align*}
    |\mu-\cos\theta| \geq \frac{1}{2}|\theta-\theta_\mu|(\sin\theta+\sin\theta_\mu) \geq \frac{1}{2}|\theta-\theta_\mu|\sin\theta_\mu, \quad \theta\in[0,\pi].
  \end{align*}
  By this inequality and Hardy's inequality we have
  \begin{align*}
    \int_{\theta_1}^{\theta_2}\left|\frac{U(\theta)\sqrt{\sin\theta}-U(\theta_\mu)\sqrt{\sin\theta_\mu}}{\mu-\cos\theta}\right|^2\,d\theta &\leq \frac{4}{\sin^2\theta_\mu}\int_{\theta_1}^{\theta_2}\left|\frac{U(\theta)\sqrt{\sin\theta}-U(\theta_\mu)\sqrt{\sin\theta_\mu}}{\theta-\theta_\mu}\right|^2\,d\theta \\
    &\leq \frac{16}{\sin^2\theta_\mu}\int_{\theta_1}^{\theta_2}\left|\frac{d}{d\theta}\Bigl(U(\theta)\sqrt{\sin\theta}\Bigr)\right|^2\,d\theta.
  \end{align*}
  To the right-hand side we further apply
  \begin{align*}
    \left|\frac{d}{d\theta}\Bigl(U(\theta)\sqrt{\sin\theta}\Bigr)\right|^2 \leq 2\left(|U'(\theta)|^2+\frac{m^2}{\sin^2\theta}|U(\theta)|^2\right)\sin\theta
  \end{align*}
  by Young's inequality and $|m|\geq1$, and then use \eqref{E:L2_Mode} to get \eqref{E:Hardy}.
\end{proof}

\section{Linear stability of the two-jet Kolmogorov type flow} \label{S:LiSt}
In this section we study the linear stability of the two-jet Kolmogorov type flow.

Let $I$ be the identity operator and $M_f$ the multiplication operator by a function $f$ on $S^2$, i.e. $M_fu=fu$ for a function $u$ on $S^2$. We define linear operators $A$ and $\Lambda$ on $L_0^2(S^2)$ by
\begin{alignat*}{2}
  A &= \Delta+2, &\quad D_{L_0^2(S^2)}(A) &= L_0^2(S^2)\cap H^2(S^2), \\
  \Lambda &= -i\partial_\varphi M_{\cos\theta}B, &\quad D_{L_0^2(S^2)}(\Lambda) &= \{u\in L_0^2(S^2) \mid \partial_\varphi M_{\cos\theta}u\in L_0^2(S^2)\},
\end{alignat*}
where $B=I+6\Delta^{-1}$ on $L_0^2(S^2)$. For $n\geq1$ and $|m|\leq n$, since
\begin{align} \label{E:Y_dphi}
  \Delta Y_n^m = -\lambda_nY_n^m, \quad BY_n^m = \left(1-\frac{6}{\lambda_n}\right)Y_n^m, \quad \partial_\varphi Y_n^m = imY_n^m,
\end{align}
we observe by these equalities and \eqref{E:Y_Rec} that
\begin{align} \label{E:Y_A_Lam}
  AY_n^m = (-\lambda_n+2)Y_n^m, \quad \Lambda Y_n^m = m\left(1-\frac{6}{\lambda_n}\right)(a_n^mY_{n-1}^m+a_{n+1}^mY_{n+1}^m)
\end{align}
with $Y_{|m|-1}^m\equiv0$. In particular, $\Lambda Y_2^m=0$ for $|m|=0,1,2$ by $\lambda_2=6$. By this fact, we also find that $(\Lambda u,Y_1^m)_{L^2(S^2)}=0$ for $u\in D_{L_0^2(S^2)}(\Lambda)$ and $|m|=0,1$.

The operator $A$ is self-adjoint and has a compact resolvent in $L_0^2(S^2)$ since $\Delta$ does so. Moreover, $-A$ is nonnegative in $L_0^2(S^2)$. Indeed, since $\lambda_1=2$ and $\lambda_n\geq\lambda_2=6$ for $n\geq2$,
\begin{align} \label{E:Ko_A_nz}
  (-Au,u)_{L^2(S^2)} = \sum_{n=1}^\infty\sum_{m=-n}^n(\lambda_n-2)|(u,Y_n^m)_{L^2(S^2)}|^2 \geq 4\|u_{\geq2}\|_{L^2(S^2)}^2
\end{align}
for $u\in D_{L_0^2(S^2)}(A)$ with $u_{\geq2}=\sum_{n=2}^\infty\sum_{m=-n}^n(u,Y_n^m)_{L^2(S^2)}Y_n^m$.

The operator $\Lambda$ is densely defined in $L_0^2(S^2)$ since its domain contains the dense subset $L_0^2(S^2)\cap H^1(S^2)$ of $L_0^2(S^2)$. Moreover, $\Lambda$ is closed in $L_0^2(S^2)$ since $M_{\cos\theta}B$ is a bounded operator on $L^2(S^2)$ (note that it does not map $L_0^2(S^2)$ into itself) and $\partial_\varphi$ is a closed operator from $L^2(S^2)$ into $L_0^2(S^2)$. We also observe that $\Lambda$ is $A$-compact in $L_0^2(S^2)$ since $H^2(S^2)$ is compactly embedded into $H^1(S^2)$ and
\begin{align*}
  \|u\|_{H^2(S^2)} \leq C\|\Delta u\|_{L^2(S^2)} \leq C\left(\|Au\|_{L^2(S^2)}+2\|u\|_{L^2(S^2)}\right), \quad u\in D_{L_0^2(S^2)}(A)
\end{align*}
by the elliptic regularity theorem.

For $\nu>0$ and $a\in\mathbb{R}$ let $\mathcal{L}^{\nu,a}$ be a linear operator on $L_0^2(S^2)$ given by
\begin{align*}
  \mathcal{L}^{\nu,a} = \nu A-ia\Lambda, \quad D_{L_0^2(S^2)}(\mathcal{L}^{\nu,a}) = D_{L_0^2(S^2)}(A).
\end{align*}
By a perturbation theory of semigroups (see \cite[Section III.2]{EngNag00}), $\mathcal{L}^{\nu,a}$ generates an analytic semigroup $\{e^{t\mathcal{L}^{\nu,a}}\}_{t\geq0}$ in $L_0^2(S^2)$. As mentioned in Section \ref{S:Intro}, $e^{t\mathcal{L}^{\nu,a}}$ is the solution operator of the linearized equation for \eqref{E:Vo_Intro} around the stationary solution $\omega_2^a$  of the form \eqref{E:Zna_Intro} with $n=2$. Our aim is to show the following result on the linear stability of $\omega_2^a$, which is a precise form of Theorem \ref{T:Li2_LS}. In what follows, we write $u_{=n}=\sum_{m=-n}^n(u,Y_n^m)_{L^2(S^2)}Y_n^m$ and $u_{\geq N}=\sum_{n\geq N}u_{=n}$ for $u\in L_0^2(S^2)$ and $n,N\in\mathbb{N}$.

\begin{theorem} \label{T:Li2S_Pre}
  For $\nu>0$, $a\in\mathbb{R}$, and $\omega_0\in L_0^2(S^2)$, let $\Omega_0^{\nu,a}$ be given by \eqref{E:Li2S_Om}. Also, let $\tilde{\omega}(t)=e^{t\mathcal{L}^{\nu,a}}\omega_0-\Omega_0^{\nu,a}$ for $t\geq0$. Then $\tilde{\omega}_{=1}(t)=0$ and
  \begin{align} \label{E:L2P_g3}
    \|\tilde{\omega}_{\geq3}(t)\|_{L^2(S^2)} \leq e^{-10\nu t}\|\omega_{0,\geq3}\|_{L^2(S^2)}
  \end{align}
  for all $t\geq0$. Moreover, $(\tilde{\omega}(t),Y_2^0)_{L^2(S^2)}=e^{-4\nu t}(\omega_0,Y_2^0)_{L^2(S^2)}$ and
  \begin{align} \label{E:L2P_e2}
    \bigl|(\tilde{\omega}(t),Y_2^m)_{L^2(S^2)}\bigr| \leq Ce^{-4\nu t}\left\{\bigl|(\omega_0,Y_2^m)_{L^2(S^2)}\bigr|+\frac{|a|}{\nu}R_m(\omega_0)\right\}
  \end{align}
  for all $t\geq0$ and $|m|=1,2$, where
  \begin{align*}
    R_{\pm1}(\omega_0) = \bigl|(\omega_0,Y_1^{\pm1})_{L^2(S^2)}\bigr|+\|\omega_{0,\geq3}\|_{L^2(S^2)}, \quad R_{\pm2}(\omega_0) = \|\omega_{0,\geq3}\|_{L^2(S^2)}
  \end{align*}
  and $C>0$ is a constant independent of $t$, $\nu$, $a$, $\omega_0$, and $m$.
\end{theorem}

\begin{proof}
  Let $\tilde{\omega}(t)=\sum_{n=1}^\infty\sum_{m=-n}^n\tilde{c}_n^m(t)Y_n^m$ with $\tilde{c}_n^m(t)=(\tilde{\omega}(t),Y_n^m)_{L^2(S^2)}$. We see by \eqref{E:Li2S_Om} and \eqref{E:Y_A_Lam} that $\mathcal{L}^{\nu,a}\Omega_0^{\nu,a}=0$. Hence $\tilde{\omega}(t)$ satisfies
  \begin{align} \label{Pf_L2P:Om_Eq}
    \partial_t\tilde{\omega}(t) = \nu A\tilde{\omega}(t)-ia\Lambda\tilde{\omega}(t), \quad t>0, \quad \tilde{\omega}(0) = \omega_0-\Omega_0^{\nu,a}.
  \end{align}
  For $|m|=0,1$ we take the $L^2(S^2)$-inner product of \eqref{Pf_L2P:Om_Eq} with $Y_1^m$ to get
  \begin{align*}
    \frac{d\tilde{c}_1^m}{dt}(t) = \nu(A\tilde{\omega}(t),Y_1^m)_{L^2(S^2)}-ia(\Lambda\tilde{\omega}(t),Y_1^m)_{L^2(S^2)} = 0, \quad t>0
  \end{align*}
  by \eqref{E:Y_A_Lam}. Hence $\tilde{c}_1^m(t)=\tilde{c}_1^m(0)=0$ by \eqref{E:Li2S_Om} and $\tilde{\omega}_{=1}(t)=0$ for $t\geq0$.

  Next we prove \eqref{E:L2P_g3}. Let $u(t)=B\tilde{\omega}(t)$. Then
  \begin{align} \label{Pf_L2P:u}
    u(t) = \sum_{n=3}^\infty\sum_{m=-n}^n\left(1-\frac{6}{\lambda_n}\right)\tilde{c}_n^m(t)Y_n^m, \quad \|u(t)\|_{L^2(S^2)} \leq \|\tilde{\omega}_{\geq3}(t)\|_{L^2(S^2)}
  \end{align}
  by \eqref{E:Y_dphi}, $\tilde{\omega}_{=1}(t)=0$, $\lambda_2=6$, and $|1-6/\lambda_n|\leq1$ for $n\geq3$.
  We easily get
  \begin{align*}
    \mathrm{Re}\bigl(\partial_t\tilde{\omega}(t),u(t)\bigr)_{L^2(S^2)} &\geq \frac{1}{4}\frac{d}{dt}\|\tilde{\omega}_{\geq3}(t)\|_{L^2(S^2)}^2, \\
    \mathrm{Re}\bigl(A\tilde{\omega}(t),u(t)\bigr)_{L^2(S^2)} &\leq -5\|\tilde{\omega}_{\geq3}(t)\|_{L^2(S^2)}^2
  \end{align*}
  by \eqref{E:Y_A_Lam}, the first equality of \eqref{Pf_L2P:u}, and $\lambda_n\geq\lambda_3=12$ for $n\geq3$. Also,
  \begin{align*}
    \mathrm{Im}\bigl(\Lambda\tilde{\omega}(t),u(t)\bigr)_{L^2(S^2)} = \mathrm{Im}\bigl(-i\partial_\varphi M_{\cos\theta}u(t),u(t)\bigr)_{L^2(S^2)} = 0
  \end{align*}
  since $-i\partial_\varphi M_{\cos\theta}$ is symmetric in $L_0^2(S^2)$. Thus, taking the real part of the $L^2(S^2)$-inner product of \eqref{Pf_L2P:Om_Eq} with $u(t)$ and using the above relations, we find that
  \begin{align*}
    \frac{1}{4}\frac{d}{dt}\|\tilde{\omega}_{\geq3}(t)\|_{L^2(S^2)}^2 \leq -5\nu\|\tilde{\omega}_{\geq3}(t)\|_{L^2(S^2)}^2, \quad t>0.
  \end{align*}
  By this inequality we get \eqref{E:L2P_g3}, since $\tilde{\omega}_{\geq3}(0)=\omega_{0,\geq3}$ by \eqref{E:Li2S_Om}.

  Let us consider $\tilde{c}_2^m(t)$. For $m=0$ we take the $L^2(S^2)$-inner product of \eqref{Pf_L2P:Om_Eq} with $Y_2^0$ and use \eqref{E:Y_A_Lam} to get $\frac{d}{dt}\tilde{c}_2^0(t)=-4\nu\tilde{c}_2^0(t)$ for $t>0$. Thus $\tilde{c}_2^0(t)=e^{-4\nu t}\tilde{c}_2^0(0)=e^{-4\nu t}(\omega_0,Y_2^0)_{L^2(S^2)}$ for $t\geq0$ by \eqref{E:Li2S_Om}. Let $|m|=1,2$. Since $-i\partial_\varphi$ is symmetric in $L_0^2(S^2)$,
  \begin{align*}
    (\Lambda\tilde{\omega}(t),Y_2^m)_{L^2(S^2)} &= (-i\partial_\varphi M_{\cos\theta}u(t),Y_2^m)_{L^2(S^2)} = (M_{\cos\theta}u(t),-i\partial_\varphi Y_2^m)_{L^2(S^2)} \\
    &= m(M_{\cos\theta}u(t),Y_2^m)_{L^2(S^2)},
  \end{align*}
  where $u(t)=B\tilde{\omega}(t)$ and the last equality is due to \eqref{E:Y_dphi}. We take the $L^2(S^2)$-inner product of \eqref{Pf_L2P:Om_Eq} with $Y_2^m$ and use \eqref{E:Y_A_Lam} and the above equality. Then
  \begin{align*}
    \frac{d\tilde{c}_2^m}{dt}(t) = -4\nu\tilde{c}_2^m(t)-ima(M_{\cos\theta}u(t),Y_2^m)_{L^2(S^2)}, \quad t>0.
  \end{align*}
  We solve this equation to get
  \begin{align*}
    \tilde{c}_2^m(t) = e^{-4\nu t}\left(\tilde{c}_2^m(0)-ima\int_0^te^{4\nu\tau}(M_{\cos\theta}u(\tau),Y_2^m)_{L^2(S^2)}\,d\tau\right), \quad t\geq0.
  \end{align*}
  Moreover, we apply $\|Y_2^m\|_{L^2(S^2)}=1$, $|\cos\theta|\leq1$, \eqref{E:L2P_g3}, and \eqref{Pf_L2P:u} to the integrand of the last term and then use $\int_0^te^{-6\nu\tau}\,d\tau\leq(6\nu)^{-1}$ to obtain
  \begin{align*}
    |\tilde{c}_2^m(t)| \leq e^{-4\nu t}\left(|\tilde{c}_2^m(0)|+\frac{|ma|}{6\nu}\|\omega_{0,\geq3}\|_{L^2(S^2)}\right), \quad t\geq0.
  \end{align*}
  Here $\tilde{c}_2^m(0)$ is of the form
  \begin{align*}
    \tilde{c}_2^m(0) = (\tilde{\omega}(0),Y_2^m)_{L^2(S^2)} =
    \begin{cases}
      (\omega_0,Y_2^m)-\displaystyle\frac{a}{\nu}\frac{im}{2\sqrt{5}}(\omega_0,Y_1^m)_{L^2(S^2)}, &|m| = 1, \\
      (\omega_0,Y_2^m)_{L^2(S^2)}, &|m| = 2
    \end{cases}
  \end{align*}
  since $\Omega_0^{\nu,a}$ is given by \eqref{E:Li2S_Om}. Hence \eqref{E:L2P_e2} follows.
\end{proof}

\section{Nonexistence of nonzero eigenvalues of the perturbation operator} \label{S:No_Ei}
The aim of this section is to establish Theorem \ref{T:NoEi_Lam}.

For $m\in\mathbb{Z}$ let $\mathcal{P}_m$ be the operator given by \eqref{E:Def_Proj}. Then $\mathcal{P}_mL_0^2(S^2)$ is a closed subspace of $L_0^2(S^2)$ by \eqref{E:Proj_Bdd}. Moreover, since functions in $L_0^2(S^2)$ are expanded by $Y_n^m$, and since $\mathcal{P}_mY_n^m=Y_n^m$ and $\mathcal{P}_mY_n^{m'}=0$ for $m\neq m'$, we see that $L_0^2(S^2)$ is diagonalized as
\begin{align*}
  L_0^2(S^2) = \oplus_{m\in\mathbb{Z}}\mathcal{P}_mL_0^2(S^2)
\end{align*}
and each $u\in\mathcal{P}_mL_0^2(S^2)$, $m\in\mathbb{Z}$ is expressed as
\begin{align} \label{E:PmL02}
  u = \sum_{n\geq\max\{1,|m|\}}c_n^mY_n^m, \quad c_n^m = (u,Y_n^m)_{L^2(S^2)}.
\end{align}
We observe by \eqref{E:Y_A_Lam} and \eqref{E:PmL02} that $\mathcal{P}_mL_0^2(S^2)$ is invariant under the action of $\Lambda$ for each $m\in\mathbb{Z}$. Moreover, $\Lambda$ is diagonalized as
\begin{align} \label{E:Lam_diag}
  \Lambda = \oplus_{m\in\mathbb{Z}}\Lambda|_{\mathcal{P}_mL_0^2(S^2)}, \quad \Lambda|_{\mathcal{P}_mL_0^2(S^2)} =
  \begin{cases}
    0, &m=0, \\
    m\Lambda_m, &m\neq0,
  \end{cases}
\end{align}
where $\Lambda_m=M_{\cos\theta}B|_{\mathcal{P}_mL_0^2(S^2)}$ on $\mathcal{P}_mL_0^2(S^2)$. Thus, to prove Theorem \ref{T:NoEi_Lam}, it is sufficient to show that $\Lambda_m$ does not admit nonzero eigenvalues for each $m\in\mathbb{Z}\setminus\{0\}$. Note that
\begin{align} \label{E:Y_Lam_Pm}
  \Lambda_mY_n^m = \left(1-\frac{6}{\lambda_n}\right)(a_n^mY_{n-1}^m+a_{n+1}^mY_{n+1}^m), \quad n \geq |m|
\end{align}
by \eqref{E:Y_Rec} and \eqref{E:Y_dphi}. In particular, $\Lambda_mY_2^m=0$ for $|m|=1,2$ by $\lambda_2=6$.

\begin{theorem} \label{T:NoEi_Pm}
  Let $m\in\mathbb{Z}\setminus\{0\}$. Then $\Lambda_m$ in $\mathcal{P}_mL_0^2(S^2)$ has no eigenvalues in $\mathbb{C}\setminus\{0\}$.
\end{theorem}

\begin{proof}
  Suppose that $\mu\in\mathbb{C}\setminus\{0\}$ and $u\in\mathcal{P}_mL_0^2(S^2)$ satisfy
  \begin{align} \label{Pf_NEm:Eq}
    \mu u = \Lambda_mu = M_{\cos\theta}Bu, \quad B = I+6\Delta^{-1}.
  \end{align}
  Let us show $u=0$. First we observe that $u$ is of the form
  \begin{align} \label{Pf_NEm:u_Form}
    u = \sum_{n\geq N_m}c_n^mY_n^m, \quad N_m = \max\{2,|m|\}.
  \end{align}
  When $|m|\geq2$, this is the same as the expression \eqref{E:PmL02} of $u\in\mathcal{P}_mL_0^2(S^2)$. If $m=\pm1$, then it follows from \eqref{E:PmL02} and \eqref{E:Y_Lam_Pm} (in particular $\Lambda_{\pm1}Y_2^{\pm1}=0$) that
  \begin{align*}
    \mu c_1^{\pm1} = (\mu u,Y_1^{\pm1})_{L^2(S^2)} = (\Lambda_{\pm1}u,Y_1^{\pm1})_{L^2(S^2)} = 0.
  \end{align*}
  Hence $c_1^{\pm1}=0$ by $\mu\neq0$ and $u$ is of the form \eqref{Pf_NEm:u_Form}.

  Suppose that $\mathrm{Im}\,\mu\neq0$. We take the imaginary part of the $L^2(S^2)$-inner product of \eqref{Pf_NEm:Eq} with $Bu$. Then since $u$ is of the form \eqref{Pf_NEm:u_Form} and $M_{\cos\theta}$ is symmetric on $L^2(S^2)$, we see by \eqref{E:Y_dphi} and $\lambda_2=6$ that
  \begin{align*}
    (\mathrm{Im}\,\mu)\sum_{n\geq N_m'}\left(1-\frac{6}{\lambda_n}\right)|c_n^m|^2 = \mathrm{Im}(M_{\cos\theta}Bu,Bu)_{L^2(S^2)} = 0,
  \end{align*}
  where $N_m'=\max\{3,|m|\}$. Thus, by $\mathrm{Im}\,\mu\neq0$ and $1-6/\lambda_n\geq1/2$ for $n\geq3$,
  \begin{align*}
    0 = \sum_{n\geq N_m'}\left(1-\frac{6}{\lambda_n}\right)|c_n^m|^2 \geq \frac{1}{2}\sum_{n\geq N_m'}|c_n^m|^2,
  \end{align*}
  which shows that $c_n^m=0$ for $n\geq N_m'$, i.e. $u=c_2^mY_2^m$ if $|m|=1,2$ and $u=0$ if $|m|\geq3$. Moreover, when $|m|=1,2$, we have $\mu u=\Lambda_mu=c_2^m\Lambda_mY_2^m=0$ and thus $u=0$ by $\mu\neq0$. Hence we get $u=0$ in both cases $|m|=1,2$ and $|m|\geq3$ when $\mathrm{Im}\,\mu\neq0$.

  Now suppose that $\mathrm{Im}\,\mu=0$, i.e. $\mu\in\mathbb{R}\setminus\{0\}$. We consider two cases separately.

  \textbf{Case 1:} $|\mu|\geq1$. Since $u$ is of the form \eqref{Pf_NEm:u_Form}, we see that
  \begin{align*}
    \mu^2\sum_{n\geq N_m}|c_n^m|^2 = \|\mu u\|_{L^2(S^2)}^2 = \|\Lambda_mu\|_{L^2(S^2)}^2 \leq \|Bu\|_{L^2(S^2)}^2 = \sum_{n\geq N_m}\left(1-\frac{6}{\lambda_n}\right)^2|c_n^m|^2
  \end{align*}
  by \eqref{E:Y_dphi}, \eqref{Pf_NEm:Eq}, and $|\cos\theta|\leq 1$. Hence
  \begin{align*}
    \sum_{n\geq N_m}\left\{\mu^2-\left(1-\frac{6}{\lambda_n}\right)^2\right\}|c_n^m|^2 \leq 0,
  \end{align*}
  but since $0\leq1-6/\lambda_n<1$ for $n\geq2$ and $\mu^2\geq1$, we must have $c_n^m=0$ for all $n\geq N_m$ by the above inequality. Therefore, we obtain $u=0$.

  \textbf{Case 2:} $0<|\mu|<1$. Since $u$ is of the form \eqref{Pf_NEm:u_Form}, we write
  \begin{align} \label{Pf_NEm:u_dec}
    u = \sum_{n=N_m}^{N-1}c_n^mY_n^m+u_{\geq N}, \quad u_{\geq N} = \sum_{n\geq N}c_n^mY_n^m, \quad N_m = \max\{2,|m|\},
  \end{align}
  where $N>N_m$ is a sufficiently large integer which will be fixed later. We substitute \eqref{Pf_NEm:u_dec} for \eqref{Pf_NEm:Eq} and use \eqref{E:Y_dphi} to get
  \begin{align*}
    \sum_{n=N_m}^{N-1}\mu c_n^mY_n^m+\mu u_{\geq N} = \cos\theta\left\{\sum_{n=N_m}^{N-1}\left(1-\frac{6}{\lambda_n}\right)c_n^mY_n^m+Bu_{\geq N}\right\}.
  \end{align*}
  We deduce from this equality and $u_{\geq N}=Bu_{\geq N}-6\Delta^{-1}u_{\geq N}$ that
  \begin{multline} \label{Pf_NEm:w_Eq}
    (\mu-\cos\theta)\left\{\sum_{n=N_m}^{N-1}\left(1-\frac{6}{\lambda_n}\right)c_n^mY_n^m+Bu_{\geq N}\right\} \\
    = -\sum_{n=N_m}^{N-1}\frac{6}{\lambda_n}\mu c_n^mY_n^m+6\mu\Delta^{-1}u_{\geq N}.
  \end{multline}
  Moreover, we observe by \eqref{E:Y_Rec} with $n=|m|$ and $Y_{|m|-1}^m\equiv0$ that
  \begin{align*}
    Y_{|m|+1}^m = \frac{1}{a_{|m|+1}^m}\cos\theta\,Y_{|m|}^m = -\frac{1}{a_{|m|+1}^m}(\mu-\cos\theta)Y_{|m|}^m+\frac{\mu}{a_{|m|+1}^m}Y_{|m|}^m,
  \end{align*}
  and by \eqref{E:Y_Rec} with $n$ replaced by $n-1$ that
  \begin{align*}
    Y_n^m = \frac{1}{a_n^m}(\cos\theta\,Y_{n-1}^m-a_{n-1}^mY_{n-2}^m) = -\frac{1}{a_n^m}(\mu-\cos\theta)Y_{n-1}^m+\frac{1}{a_n^m}(\mu Y_{n-1}^m-a_{n-1}^mY_{n-2}^m)
  \end{align*}
  for $n\geq|m|+2$. Using these equalities, we can inductively show that
  \begin{align} \label{Pf_NEm:Y_Red}
    Y_n^m = (\mu-\cos\theta)\left(\sum_{k=|m|}^{n-1}\alpha_{k,n}^mY_k^m\right)+\beta_n^mY_{|m|}^m, \quad n\geq |m|+1.
  \end{align}
  Here $\alpha_{k,n}^m$ and $\beta_n^m$ are some coefficients depending on $\mu$ and $a_{n'}^m$ with $n'\geq|m|$, but we do not need their explicit forms. We substitute \eqref{Pf_NEm:Y_Red} for the right-hand side of \eqref{Pf_NEm:w_Eq}. Then, moving the terms with the factor $\mu-\cos\theta$ into the left-hand side, we have
  \begin{align} \label{Pf_NEm:f_Eq}
    (\mu-\cos\theta)\left(\sum_{n=|m|}^{N-1}\gamma_{n,N}^mY_n^m+Bu_{\geq N}\right) = \sigma_N^mY_{|m|}^m+6\mu\Delta^{-1}u_{\geq N}.
  \end{align}
  Here $\gamma_{n,N}^m$ and $\sigma_N^m$ are some coefficients depending on $\mu$, $\lambda_{n'}$, $a_{n'}^m$, and $c_{n'}^m$ with $n'\geq|m|$, but again we do not need their explicit forms. Since $\Delta^{-1}u_{\geq N}\in\mathcal{P}_mL_0^2(S^2)\cap H^2(S^2)$ and the Sobolev embedding $H^2(S^2)\hookrightarrow C(S^2)$ holds (see \cite{Aub98}), we can write
  \begin{align} \label{Pf_NEm:UN}
    \Delta^{-1}u_{\geq N} = \widetilde{U}_N(\theta)e^{im\varphi} \in \mathcal{P}_mL_0^2(S^2)\cap C(S^2), \quad \widetilde{U}_N \in C([0,\pi]).
  \end{align}
  Also, $Y_{|m|}^{m}$ is smooth on $S^2$ and of the form (see \eqref{E:SpHa})
  \begin{align*}
    Y_{|m|}^{m}(\theta,\varphi) = C_m\sin^{|m|}\theta\,e^{im\varphi} = C_m(1-\cos^2\theta)^{|m|/2}\,e^{im\varphi}
  \end{align*}
  with a constant $C_m\in\mathbb{R}\setminus\{0\}$. Thus $f_N^m=\sigma_N^mY_{|m|}^m+6\mu\Delta^{-1}u_{\geq N}$ is of the form
  \begin{align*}
    f_N^m = F_N^m(\theta)e^{im\varphi} \in \mathcal{P}_mL_0^2(S^2)\cap H^2(S^2), \quad F_N^m(\theta) = \sigma_N^mC_m(1-\cos^2\theta)^{|m|/2}+6\mu\widetilde{U}_N(\theta).
  \end{align*}
  Moreover, since $f_N^m$ is continuous on $S^2$ by $H^2(S^2)\subset C(S^2)$ and
  \begin{align} \label{Pf_NEm:f_Div}
    \sum_{n=|m|}^{N-1}\gamma_{n,N}^mY_n^m+Bu_{\geq N} = \frac{f_N^m}{\mu-x_3} \quad\text{on}\quad \{(x_1,x_2,x_3)\in S^2 \mid x_3\neq\mu\}
  \end{align}
  by \eqref{Pf_NEm:f_Eq} and $x_3=\cos\theta$, we have $f_N^m=0$ for $x_3=\mu$, otherwise the left-hand side of \eqref{Pf_NEm:f_Div} does not belong to $L^2(S^2)$. Hence $F_N^m(\theta_\mu)=0$ with $\theta_\mu=\arccos\mu\in(0,\pi)$, i.e.
  \begin{align*}
    \sigma_N^mC_m(1-\mu^2)^{|m|/2}+6\mu\widetilde{U}_N(\theta_\mu) = 0, \quad |\sigma_N^m|^2 = \frac{(6\mu)^2}{C_m^2(1-\mu^2)^{|m|}}\left|\widetilde{U}_N(\theta_\mu)\right|^2.
  \end{align*}
  Here and in the rest of the proof, $C_m^2$ stands for the square of $C_m$, not for a coefficient of $Y_m^2$. Moreover, noting that $\widetilde{U}_N$ is given by \eqref{Pf_NEm:UN}, we use \eqref{E:Linf} to get
  \begin{align} \label{Pf_NEm:sigma}
    |\sigma_N^m|^2 \leq \frac{(6\mu)^2}{\pi|m|C_m^2(1-\mu^2)^{|m|}}\|(-\Delta)^{-1/2}u_{\geq N}\|_{L^2(S^2)}^2.
  \end{align}
  Since $f_N^m=F_N^m(\theta)e^{im\varphi}$ and $F_N^m(\theta_\mu)=0$, we see by \eqref{E:L2_Mode} and \eqref{E:Hardy} that
  \begin{align*}
    \left\|\frac{f_N^m}{\mu-x_3}\right\|_{L^2(S^2)}^2 = 2\pi\int_0^\pi\left|\frac{F_N^m(\theta)\sqrt{\sin\theta}}{\mu-\cos\theta}\right|^2\,d\theta &\leq \frac{32}{\sin^2\theta_\mu}\|\nabla f_N^m\|_{L^2(S^2)}^2
  \end{align*}
  and use \eqref{E:Lap_Half} and $\sin^2\theta_\mu=1-\mu^2$ to the right-hand side to get
  \begin{align} \label{Pf_NEm:fd_Lhf}
    \left\|\frac{f_N^m}{\mu-x_3}\right\|_{L^2(S^2)}^2 \leq \frac{32}{1-\mu^2}\|(-\Delta)^{1/2}f_N^m\|_{L^2(S^2)}^2.
  \end{align}
  Moreover, since
  \begin{align*}
    (-\Delta)^{1/2}f_N^m = (-\Delta)^{1/2}(\sigma_N^mY_{|m|}^m+6\mu\Delta^{-1}u_{\geq N}) = \sigma_N^m\lambda_{|m|}^{1/2}Y_{|m|}^m-6\mu(-\Delta)^{-1/2}u_{\geq N}
  \end{align*}
  and $(-\Delta)^{-1/2}u_{\geq N}$ is orthogonal to $Y_{|m|}^m$ in $L^2(S^2)$ by \eqref{E:Def_Laps} and \eqref{Pf_NEm:u_dec},
  \begin{align} \label{Pf_NEm:L2_Lhf}
    \begin{aligned}
      \|(-\Delta)^{1/2}f_N^m\|_{L^2(S^2)}^2 &= |\sigma_N^m|^2\lambda_{|m|}\|Y_{|m|}^m\|_{L^2(S^2)}^2+(6\mu)^2\|(-\Delta)^{-1/2}u_{\geq N}\|_{L^2(S^2)}^2 \\
      &\leq (6\mu)^2\left\{\frac{|m|+1}{\pi C_m^2(1-\mu^2)^{|m|}}+1\right\}\|(-\Delta)^{-1/2}u_{\geq N}\|_{L^2(S^2)}^2
    \end{aligned}
  \end{align}
  by $\lambda_{|m|}=|m|(|m|+1)$, $\|Y_{|m|}^m\|_{L^2(S^2)}=1$, and \eqref{Pf_NEm:sigma}. Hence
  \begin{align} \label{Pf_NEm:L2_fd}
    \left\|\frac{f_N^m}{\mu-x_3}\right\|_{L^2(S^2)}^2 \leq C_{m,\mu}\|(-\Delta)^{-1/2}u_{\geq N}\|_{L^2(S^2)}^2
  \end{align}
  by \eqref{Pf_NEm:fd_Lhf} and \eqref{Pf_NEm:L2_Lhf}, where
  \begin{align*}
    C_{m,\mu} = \frac{32\cdot(6\mu)^2}{1-\mu^2}\left\{\frac{|m|+1}{\pi C_m^2(1-\mu^2)^{|m|}}+1\right\}.
  \end{align*}
  Now we take the $L^2(S^2)$-inner product of \eqref{Pf_NEm:f_Div} with $Bu_{\geq N}$. Then since $Bu_{\geq N}$ is orthogonal to $Y_{|m|}^m,\dots,Y_{N-1}^m$ in $L^2(S^2)$ by \eqref{E:Y_dphi} and \eqref{Pf_NEm:u_dec}, we get
  \begin{align*}
    \|Bu_{\geq N}\|_{L^2(S^2)}^2 = \left(\frac{f_N^m}{\mu-x_3},Bu_{\geq N}\right)_{L^2(S^2)} \leq \left\|\frac{f_N^m}{\mu-x_3}\right\|_{L^2(S^2)}\|Bu_{\geq N}\|_{L^2(S^2)}.
  \end{align*}
  We deduce from this inequality and \eqref{Pf_NEm:L2_fd} that
  \begin{align*}
    \|Bu_{\geq N}\|_{L^2(S^2)}^2 \leq \left\|\frac{f_N^m}{\mu-x_3}\right\|_{L^2(S^2)}^2 \leq C_{m,\mu}\|(-\Delta)^{-1/2}u_{\geq N}\|_{L^2(S^2)}^2.
  \end{align*}
  Moreover, when $N\geq3$, we have
  \begin{align*}
    \|Bu_{\geq N}\|_{L^2(S^2)}^2 \geq \frac{1}{4}\|u_{n\geq N}\|_{L^2(S^2)}^2, \quad \|(-\Delta)^{-1/2}u_{\geq N}\|_{L^2(S^2)}^2 \leq \frac{1}{\lambda_N}\|u_{n\geq N}\|_{L^2(S^2)}^2
  \end{align*}
  by \eqref{E:Lap_Half}, \eqref{E:Y_dphi}, \eqref{Pf_NEm:u_dec}, $1-6/\lambda_n\geq1/2$, and $\lambda_n\geq\lambda_N$ for $n\geq N$. Hence
  \begin{align} \label{Pf_NEm:Bw}
    \|u_{\geq N}\|_{L^2(S^2)}^2 \leq \frac{4C_{m,\mu}}{\lambda_N}\|u_{\geq N}\|_{L^2(S^2)}^2.
  \end{align}
  Now since $\lambda_N=N(N+1)\to\infty$ as $N\to\infty$ and $C_{m,\mu}$ is independent of $N$, we can fix a sufficiently large $N>N_m$ so that $4C_{m,\mu}/\lambda_N<1$. Then $u_{\geq N}=0$ by \eqref{Pf_NEm:Bw} and
  \begin{align*}
    u = \sum_{n=N_m}^{N-1}c_n^mY_n^m+u_{\geq N} = \sum_{n=N_m}^{N-1}c_n^mY_n^m, \quad N_m = \max\{2,|m|\}
  \end{align*}
  by \eqref{Pf_NEm:u_dec}. By this fact and \eqref{E:Y_Lam_Pm}, the equation $\mu u=\Lambda_mu$ reads
  \begin{align*}
    \sum_{n=N_m}^{N-1}\mu c_n^mY_n^m = \sum_{n=N_m}^{N-1}\left(1-\frac{6}{\lambda_n}\right)c_n^m(a_n^mY_{n-1}^m+a_{n+1}^mY_{n+1}^m).
  \end{align*}
  We get $0=(1-6/\lambda_{N-1})c_{N-1}^ma_N^m$ by taking the $L^2(S^2)$-inner product of the above equality with $Y_N^m$. Thus $c_{N-1}^m=0$ by $1-6/\lambda_{N-1}\neq0$ and $a_N^m\neq0$. We also have
  \begin{align*}
    c_{N-2}^m = \dots = c_{N_m+1}^m = 0, \quad\text{i.e.}\quad u = c_{N_m}^mY_{N_m}^m
  \end{align*}
  by repeating the above arguments, and then find that
  \begin{align*}
    \mu c_{N_m}^mY_{N_m}^m = \mu u = \Lambda_mu =
    \begin{cases}
      0, &|m|=1,2, \\
      \left(1-\displaystyle\frac{6}{\lambda_{N_m}}\right)c_{N_m}^ma_{N_m+1}^mY_{N_m+1}^m, &|m|\geq3,
    \end{cases}
  \end{align*}
  which yields $\mu c_{N_m}^m=0$ and thus $c_{N_m}^m=0$ by $\mu\neq0$. Hence we get $u=0$ and the proof is complete.
\end{proof}

\begin{proof}[Proof of Theorem \ref{T:NoEi_Lam}]
  Suppose that $\Lambda u=\mu u$ for $\mu\in\mathbb{C}\setminus\{0\}$ and $u\in D_{L_0^2(S^2)}(\Lambda)$. Then we have $0=\mu\mathcal{P}_0u$ and $m\Lambda_m\mathcal{P}_mu = \mu\mathcal{P}_mu$ for $m\neq0$ since $\Lambda$ is diagonalized as \eqref{E:Lam_diag}. By the first equality and $\mu\neq0$ we get $\mathcal{P}_0u=0$. Also, for $m\neq0$, it follows from the second equality and $\mu\neq0$ that
  \begin{align*}
    \Lambda_m\mathcal{P}_mu = \frac{\mu}{m}\mathcal{P}_mu, \quad \mathcal{P}_mu \in \mathcal{P}_mL_0^2(S^2), \quad \frac{\mu}{m} \in \mathbb{C}\setminus\{0\}
  \end{align*}
  and thus $\mathcal{P}_mu=0$ by Theorem \ref{T:NoEi_Pm}. Hence $u=\sum_{m\in\mathbb{Z}}\mathcal{P}_mu=0$ and we conclude that the theorem is valid.
\end{proof}

\section{Enhanced dissipation for the rescaled flow} \label{S:EnDi}
The purpose of this section is to show that the enhanced dissipation occurs for a solution to the rescaled equation $\partial_t\omega=A\omega-i\alpha\Lambda\omega$ with $\alpha\in\mathbb{R}$.

Noting that $AY_1^m=0$ for $|m|=0,1$ and $\Lambda Y_n^0=0$ for $n\geq1$ by \eqref{E:Y_A_Lam}, we set
\begin{align*}
  \mathcal{X} = \{u\in L_0^2(S^2) \mid (u,Y_n^0)_{L^2(S^2)}=(u,Y_1^m)_{L^2(S^2)}=0, \, n\geq1, \, |m| = 0,1\}.
\end{align*}
Then $\mathcal{X}$ is a closed subspace of $L_0^2(S^2)$ and $u\in\mathcal{X}$ is expressed as
\begin{align} \label{E:Ko_u_X}
  u = \sum_{m\in\mathbb{Z}\setminus\{0\}}\sum_{n\geq N_m}c_n^mY_n^m, \quad N_m = \max\{2,|m|\}.
\end{align}
By \eqref{E:Y_A_Lam} (in particular $\Lambda Y_2^m=0$ for $|m|=1,2$) and \eqref{E:Ko_u_X}, we see that $\mathcal{X}$ is invariant under the actions of $A$ and $\Lambda$. We simply write $A$ and $\Lambda$ for their restrictions on $\mathcal{X}$ with domains
\begin{align*}
  D_{\mathcal{X}}(A) = \mathcal{X}\cap D_{L_0^2(S^2)}(A), \quad D_{\mathcal{X}}(\Lambda) = \mathcal{X}\cap D_{L_0^2(S^2)}(\Lambda).
\end{align*}
For $\alpha\in\mathbb{R}$ let $L_\alpha=A-i\alpha\Lambda$ on $\mathcal{X}$ with domain $D_{\mathcal{X}}(L_\alpha)=D_{\mathcal{X}}(A)$. We intend to apply an abstract result given in Section \ref{S:Abst} to $L_\alpha$. Let us show auxiliary lemmas.

\begin{lemma} \label{L:Ko_ALam}
  The operator $A$ is self-adjoint and has a compact resolvent in $\mathcal{X}$, and
  \begin{align} \label{E:Ko_A_Po}
    (-Au,u)_{L^2(S^2)} \geq 4\|u\|_{L^2(S^2)}^2, \quad u\in D_{\mathcal{X}}(A).
  \end{align}
  Also, $\Lambda$ is densely defined, closed, and $A$-compact in $\mathcal{X}$.
\end{lemma}

\begin{proof}
  The statements except for \eqref{E:Ko_A_Po} hold since they are valid in $L_0^2(S^2)$ and $\mathcal{X}$ is invariant under the actions of $A$ and $\Lambda$. Also, \eqref{E:Ko_A_Po} follows from \eqref{E:Ko_A_nz} and \eqref{E:Ko_u_X}.
\end{proof}

\begin{lemma} \label{L:Ko_KerL}
  For $u\in\mathcal{X}$ let $\mathbb{Q}u=u-\sum_{|m|=1,2}(u,Y_2^m)_{L^2(S^2)}Y_2^m$. Then
  \begin{align} \label{E:B2_LowB}
    \frac{1}{2}\|\mathbb{Q}u\|_{L^2(S^2)} \leq \|Bu\|_{L^2(S^2)} \leq \|u\|_{L^2(S^2)}, \quad u\in\mathcal{X},
  \end{align}
  where $B=I+6\Delta^{-1}$ on $L_0^2(S^2)$, and the kernel of $\Lambda$ in $\mathcal{X}$ is
  \begin{align} \label{E:Ko_KerL}
    N_{\mathcal{X}}(\Lambda) = \mathrm{span}\{Y_2^m\mid |m|=1,2\}.
  \end{align}
  Thus $\mathbb{Q}$ is the orthogonal projection from $\mathcal{X}$ onto
  \begin{align*}
    \mathcal{Y} = N_{\mathcal{X}}(\Lambda)^\perp=\{u\in\mathcal{X}\mid (u,Y_2^m)_{L^2(S^2)}=0,\,|m|=1,2\}.
  \end{align*}
  Moreover, $\mathbb{Q}A\subset A\mathbb{Q}$ in $\mathcal{X}$.
\end{lemma}

\begin{proof}
  Let $u\in\mathcal{X}$ be of the form \eqref{E:Ko_u_X}. Then since
  \begin{align} \label{Pf_KNL:Qu}
    \mathbb{Q}u = \sum_{m\in\mathbb{Z}\setminus\{0\}}\sum_{n\geq N'_m}c_n^mY_n^m, \quad N'_m = \max\{3,|m|\},
  \end{align}
  we have \eqref{E:B2_LowB} by \eqref{E:Y_dphi}, $\lambda_2=6$, and $1/2\leq1-6/\lambda_n\leq1$ for $n\geq3$. Let $u\in D_{\mathcal{X}}(\Lambda)$ satisfy $\Lambda u=0$. Then by integration by parts and \eqref{E:Y_dphi} we have
  \begin{align*}
    0 = (\Lambda u,Y_n^m)_{L^2(S^2)} = (f,-i\partial_\varphi Y_n^m)_{L^2(S^2)} = m(f,Y_n^m)_{L^2(S^2)},
  \end{align*}
  where $f=M_{\cos\theta}Bu$. Hence $(f,Y_n^m)_{L^2(S^2)}=0$ for $n\geq0$ and $|m|\neq0$. Also, since $u$ is of the form \eqref{E:Ko_u_X}, we see by \eqref{E:Y_Rec} and \eqref{E:Y_dphi} that $(f,Y_n^0)_{L^2(S^2)}=0$ for $n\geq0$ and thus $f=0$. Hence $Bu=0$ and $\mathbb{Q}u=0$, i.e. $u\in\mathrm{span}\{Y_2^m \mid |m|=1,2\}$ by \eqref{E:B2_LowB}. By this fact and $\Lambda Y_2^m=0$ for $|m|=1,2$, we get \eqref{E:Ko_KerL}. We also have $\mathbb{Q}A\subset A\mathbb{Q}$ in $\mathcal{X}$ since $Y_2^m$ is smooth on $S^2$ and $AY_2^m=-4Y_2^m$ for $|m|=1,2$ by \eqref{E:Y_A_Lam}, and since $A$ is self-adjoint in $\mathcal{X}$.
\end{proof}

\begin{lemma} \label{L:Ko_As3}
  Let $B_1$ be a linear operator on $\mathcal{H}=L^2(S^2)$ given by
  \begin{align*}
    B_1 = -i\partial_\varphi M_{\cos\theta}, \quad D_{\mathcal{H}}(B_1) = \{u\in L^2(S^2) \mid \partial_\varphi M_{\cos\theta}u \in L^2(S^2)\}
  \end{align*}
  and $B_2=B|_{\mathcal{X}}=(I+6\Delta^{-1})|_{\mathcal{X}}$ on $\mathcal{X}$. Then $A$ and $\Lambda$ satisfy Assumption \ref{As:La_02}.
\end{lemma}

Note that, as in \eqref{E:Def_Laps}, the operator $(-A)^s$ is defined on $\mathcal{X}$ by
\begin{align} \label{E:Def_mAs}
  (-A)^su = \sum_{m\in\mathbb{Z}\setminus\{0\}}\sum_{n\geq N_m}(\lambda_n-2)^sc_n^mY_n^m
\end{align}
for $s\in\mathbb{R}$ and $u\in\mathcal{X}$ of the form \eqref{E:Ko_u_X}.

\begin{proof}
  By the definitions and Lemma \ref{L:Ko_KerL}, we easily find that $B_1$ is a closed symmetric operator on $\mathcal{H}$, $B_2$ is a bounded self-adjoint operator on $\mathcal{X}$, and the conditions (i) and (ii) of Assumption \ref{As:La_02} are satisfied. Also, since $u=\mathbb{Q}u\in\mathcal{Y}$ is of the form \eqref{Pf_KNL:Qu}, we observe by \eqref{E:Y_dphi}, \eqref{E:Y_A_Lam}, \eqref{E:Def_mAs}, and $1-6/\lambda_n\geq1/2$ for $n\geq3$ that the condition (iii) holds with constant $C=1/2$.
\end{proof}

Now we give the main result of this section and the proof of Theorem \ref{T:EnDi_Lna}.

\begin{theorem} \label{T:Ko_EnDi}
  The operator $L_\alpha$ generates an analytic semigroup $\{e^{tL_\alpha}\}_{t\geq0}$ in $\mathcal{X}$ for all $\alpha\in\mathbb{R}$. Moreover, for each $\tau>0$ we have
  \begin{align} \label{E:Ko_EnDi}
    \lim_{|\alpha|\to\infty}\sup_{t\geq\tau}\|\mathbb{Q}e^{tL_\alpha}\|_{\mathcal{X}\to\mathcal{X}} = 0.
  \end{align}
\end{theorem}

\begin{proof}
  By Lemma \ref{L:Ko_ALam} and a perturbation theory of semigroups (see \cite{EngNag00}), we see that $L_\alpha$ generates an analytic semigroup $\{e^{tL_\alpha}\}_{t\geq0}$ in $\mathcal{X}$.

  We intend to apply Theorem \ref{T:Phi_Conv} to get \eqref{E:Ko_EnDi}. To this end, we verify the assumptions of Theorem \ref{T:Phi_Conv}. Assumptions \ref{As:A}--\ref{As:La_02} and the condition (a) of Theorem \ref{T:Phi_Conv} are valid by Lemmas \ref{L:Ko_ALam}--\ref{L:Ko_As3}. Also, the condition (e) follows from Theorem \ref{T:NoEi_Lam}. Hence it is sufficient to show that the conditions (b)--(d) are satisfied. Let $u\in D_{\mathcal{X}}(A)$. Then
  \begin{align} \label{Pf_KED:cb_Luu}
    \begin{aligned}
      \bigl|(\Lambda u,u)_{L^2(S^2)}\bigr| &= \bigl|(\partial_\varphi M_{\cos\theta}Bu,u)_{L^2(S^2)}\bigr| = \bigl|(M_{\cos\theta}Bu,\partial_\varphi u)_{L^2(S^2)}\bigr| \\
      &\leq \|u\|_{L^2(S^2)}\|\nabla u\|_{L^2(S^2)} = \|u\|_{L^2(S^2)}\|(-\Delta)^{1/2}u\|_{L^2(S^2)}
    \end{aligned}
  \end{align}
  by integration by parts, $|\cos\theta|\leq1$, \eqref{E:Lap_Half}, and \eqref{E:B2_LowB}. Moreover, since $u$ is of the form \eqref{E:Ko_u_X} and $\lambda_n\leq3(\lambda_n-2)/2$ for $n\geq2$, we see by \eqref{E:Def_Laps} and \eqref{E:Def_mAs} that
  \begin{align} \label{Pf_KED:cb_dphi}
    \|(-\Delta)^{1/2}u\|_{L^2(S^2)}^2 \leq \frac{3}{2}\|(-A)^{1/2}u\|_{L^2(S^2)}^2 = \frac{3}{2}(-Au,u)_{L^2(S^2)}.
  \end{align}
  Hence the condition (b) follows from \eqref{E:Ko_A_Po}, \eqref{Pf_KED:cb_Luu}, and \eqref{Pf_KED:cb_dphi}. Also,
  \begin{align*}
    D_{\mathcal{X}}(A), \, N_{\mathcal{X}}(\Lambda) \subset \mathcal{X}\cap H^1(S^2) \subset D_{\mathcal{X}}(\Lambda^\ast)
  \end{align*}
  by the definition of $A$ in $\mathcal{X}$ and \eqref{E:Ko_KerL}, and thus the condition (c) holds. Let us verify the condition (d). We observe by \eqref{E:Ko_KerL} that $f\in R_{\mathcal{X}}(\Lambda)\cap N_{\mathcal{X}}(\Lambda)$ is of the form
  \begin{align} \label{Pf_KED:cd_f}
    f = \Lambda u = \sum_{|m|=1,2}d_2^mY_2^m, \quad u\in D_{\mathcal{X}}(\Lambda), \quad d_2^m\in\mathbb{C}.
  \end{align}
  Let us show $d_2^m=0$. Since $\Lambda$ is diagonalized as \eqref{E:Lam_diag}, we have
  \begin{align} \label{Pf_KED:cd_Ym}
    mM_{\cos\theta}B\mathcal{P}_mu = d_2^mY_2^m, \quad |m| = 1,2
  \end{align}
  by \eqref{Pf_KED:cd_f}, where $B=I+6\Delta^{-1}$ and $\mathcal{P}_m$ is given by \eqref{E:Def_Proj}. When $m=1$, we apply \eqref{E:Y_Rec} with $(n,m)=(1,1)$ and $Y_0^1\equiv0$ to \eqref{Pf_KED:cd_Ym} to get
  \begin{align*}
    M_{\cos\theta}B\mathcal{P}_1u = d_2^1Y_2^1 = \frac{d_2^1}{a_2^1}M_{\cos\theta}Y_1^1, \quad M_{\cos\theta}(a_2^1B\mathcal{P}_1u-d_2^1Y_1^1) = 0.
  \end{align*}
  Hence $a_2^1B\mathcal{P}_1u-d_2^1Y_1^1=0$ on $S^2$. Moreover, since $u$ is of the form \eqref{E:Ko_u_X}, $\mathcal{P}_1u=\sum_{n\geq2}c_n^1Y_n^1$. By these facts, $d_2^1=(d_2^1Y_1^1,Y_1^1)_{L^2(S^2)}$, and \eqref{E:Y_dphi}, we find that
  \begin{align*}
    d_2^1 = (a_2^1B\mathcal{P}_1u,Y_1^1)_{L^2(S^2)} = a_2^1\sum_{n\geq2}\left(1-\frac{6}{\lambda_n}\right)c_n^1(Y_n^1,Y_1^1)_{L^2(S^2)} = 0.
  \end{align*}
  We also have $d_2^{-1}=0$ in the same way. Let $m=2$. Then since $Y_2^2 = C_2\sin^2\theta\,e^{2i\varphi}$ with a nonzero constant $C_2\in\mathbb{R}$ by \eqref{E:SpHa}, we can rewrite the equation \eqref{Pf_KED:cd_Ym} as
  \begin{align*}
    2B\mathcal{P}_2u(\theta,\varphi) = d_2^2C_2\frac{\sin^2\theta}{\cos\theta}e^{2i\varphi}, \quad (\theta,\varphi)\in[0,\pi]\times[0,2\pi), \, \theta\neq\frac{\pi}{2}.
  \end{align*}
  Hence $d_2^2=0$, otherwise the left-hand side does not belong to $L^2(S^2)$. Similarly, we have $d_2^{-2}=0$ and thus $f=0$ by \eqref{Pf_KED:cd_f}, i.e. the condition (d) is valid. Therefore, we can apply Theorem \ref{T:Phi_Conv} to obtain \eqref{E:Ko_EnDi}.
\end{proof}

\begin{proof}[Proof of Theorem \ref{T:EnDi_Lna}]
  For $\nu>0$ and $a\in\mathbb{R}$ let $\alpha=a/\nu$. Then $e^{tL_\alpha}=e^{\frac{t}{\nu}\mathcal{L}^{\nu,a}}|_{\mathcal{X}}$ in $\mathcal{X}$ for $t\geq0$ since $L_\alpha=\nu^{-1}\mathcal{L}^{\nu,a}|_{\mathcal{X}}$. Hence \eqref{E:EnDi_Lna} follows from \eqref{E:Ko_EnDi}.
\end{proof}

\section{Abstract results} \label{S:Abst}

This section gives abstract results for a perturbed operator.

For a linear operator $T$ on a Banach space $\mathcal{B}$, we denote by $D_{\mathcal{B}}(T)$, $\rho_{\mathcal{B}}(T)$, and $\sigma_{\mathcal{B}}(T)$ the domain, the resolvent set, and the spectrum of $T$ in $\mathcal{B}$. Also, let $N_{\mathcal{B}}(T)$ and $R_{\mathcal{B}}(T)$ be the kernel and range of $T$ in $\mathcal{B}$. We say that $T$ is Fredholm of index zero if $R_{\mathcal{B}}(T)$ is closed in $\mathcal{B}$ and the dimensions of $N_{\mathcal{B}}(T)$ and the quotient space $\mathcal{B}/R_{\mathcal{B}}(T)$ are finite and the same, and define
\begin{align*}
  \tilde{\sigma}_{\mathcal{B}}(T) = \{\zeta\in\mathbb{C} \mid \text{$\zeta-T$ is not Fredholm of index zero}\} \subset \sigma_{\mathcal{B}}(T).
\end{align*}
Note that $\sigma_{\mathcal{B}}(T)\setminus\tilde{\sigma}_{\mathcal{B}}(T)$ is the set of all eigenvalues of $T$ of finite multiplicity. Also, $\tilde{\sigma}_{\mathcal{B}}(T+K)=\tilde{\sigma}_{\mathcal{B}}(T)$ for every $T$-compact operator $K$, since $\zeta-(T+K)$ is Fredholm of index zero if and only if $\zeta-T$ is so for each $\zeta\in\mathbb{C}$ (see \cite[Theorem IV-5.26]{Kato76}).

Let $(\mathcal{X},(\cdot,\cdot)_{\mathcal{X}})$ be a Hilbert space and $A$ and $\Lambda$ linear operators on $\mathcal{X}$. We make the following assumptions.

\begin{assumption} \label{As:A}
  The operator $A$ is self-adjoint in $\mathcal{X}$ and satisfies
  \begin{align} \label{E:Ab_A_Po}
    (-Au,u)_{\mathcal{X}} \geq C_A\|u\|_{\mathcal{X}}^2, \quad u\in D_{\mathcal{X}}(A).
  \end{align}
  with some constant $C_A>0$.
\end{assumption}

\begin{assumption} \label{As:La_01}
  The following conditions hold:
  \begin{enumerate}
    \item The operator $\Lambda$ is densely defined, closed, and $A$-compact in $\mathcal{X}$.
    \item Let $\mathcal{Y}=N_{\mathcal{X}}(\Lambda)^\perp$ be the orthogonal complement of $N_{\mathcal{X}}(\Lambda)$ in $\mathcal{X}$ and $\mathbb{Q}$ the orthogonal projection from $\mathcal{X}$ onto $\mathcal{Y}$. Then $\mathbb{Q}A\subset A\mathbb{Q}$ in $\mathcal{X}$.
  \end{enumerate}
\end{assumption}

\begin{assumption} \label{As:La_02}
  There exist a Hilbert space $(\mathcal{H},(\cdot,\cdot)_{\mathcal{H}})$, a closed symmetric operator $B_1$ on $\mathcal{H}$, and a bounded self-adjoint operator $B_2$ on $\mathcal{X}$ such that the following conditions hold:
  \begin{enumerate}
    \item The inclusion $\mathcal{X}\subset\mathcal{H}$ holds and $(u,v)_{\mathcal{X}}=(u,v)_{\mathcal{H}}$ for all $u,v\in\mathcal{X}$.
    \item The relation $N_{\mathcal{X}}(\Lambda)=N_{\mathcal{X}}(B_2)$ holds in $\mathcal{X}$ and
    \begin{align*}
      B_2u \in D_{\mathcal{H}}(B_1), \quad B_1B_2u = \Lambda u \in\mathcal{X} \quad\text{for all}\quad u\in D_{\mathcal{X}}(\Lambda).
    \end{align*}
    \item There exists a constant $C>0$ such that
    \begin{alignat}{3}
      (u,B_2u)_{\mathcal{X}} &\geq C\|u\|_{\mathcal{X}}^2, &\quad &u\in \mathcal{Y}, \label{E:uB2u} \\
      \mathrm{Re}(-Au,B_2u)_{\mathcal{X}} &\geq C\|(-A)^{1/2}u\|_{\mathcal{X}}^2, &\quad &u\in D_{\mathcal{X}}(A)\cap\mathcal{Y}. \label{E:AB2u}
    \end{alignat}
  \end{enumerate}
\end{assumption}

Note that $B_2$ is a linear operator on the original space $\mathcal{X}$, not on the auxiliary space $\mathcal{H}$. Also, the operator $B_1$ on $\mathcal{H}$ does not necessarily map $\mathcal{X}$ into itself.

By $\mathbb{Q}A\subset A\mathbb{Q}$ in Assumption \ref{As:La_01} we can consider $\mathbb{Q}A$ as a linear operator
\begin{align*}
  \mathbb{Q}A\colon D_{\mathcal{Y}}(\mathbb{Q}A) \subset \mathcal{Y}\to\mathcal{Y}, \quad D_{\mathcal{Y}}(\mathbb{Q}A) = D_{\mathcal{X}}(A)\cap\mathcal{Y}.
\end{align*}
In what follows, we use the notation $\mathcal{N}=N_{\mathcal{X}}(\Lambda)$ for simplicity. Let $\mathbb{P}=I-\mathbb{Q}$ be the orthogonal projection from $\mathcal{X}$ onto $\mathcal{N}$ (note that $\mathcal{N}$ is closed in $\mathcal{X}$ since $\Lambda$ is closed). Then $\mathbb{P}A\subset A\mathbb{P}$ and we can also consider $\mathbb{P}A$ as a linear operator
\begin{align*}
  \mathbb{P}A\colon D_{\mathcal{N}}(\mathbb{P}A) \subset \mathcal{N} \to \mathcal{N}, \quad D_{\mathcal{N}}(\mathbb{P}A) = D_{\mathcal{X}}(A)\cap\mathcal{N}.
\end{align*}
Note that $\mathbb{Q}A$ and $\mathbb{P}A$ are closed in $\mathcal{Y}$ and in $\mathcal{N}$, respectively. Also, $\mathbb{Q}\Lambda$ is $\mathbb{Q}A$-compact in $\mathcal{Y}$. For $\alpha\in\mathbb{R}$ we define a linear operator $L_\alpha$ on $\mathcal{X}$ by
\begin{align*}
  L_\alpha=A-i\alpha\Lambda, \quad D_{\mathcal{X}}(L_\alpha)=D_{\mathcal{X}}(A)
\end{align*}
and consider $\mathbb{Q}L_\alpha=\mathbb{Q}A-i\alpha\mathbb{Q}\Lambda$ on $\mathcal{Y}$ with domain $D_{\mathcal{Y}}(\mathbb{Q}L_\alpha)=D_{\mathcal{Y}}(\mathbb{Q}A)$.

Our aim is to establish an estimate for the semigroup generated by $L_\alpha$ which yields the enhanced dissipation as $|\alpha|\to\infty$ in abstract settings. Let us give auxiliary lemmas.

\begin{lemma} \label{L:PS_Res}
  Suppose that Assumptions \ref{As:A} and \ref{As:La_01} are satisfied. Then $L_\alpha$ and $\mathbb{Q}L_\alpha$ are closed in $\mathcal{X}$ and in $\mathcal{Y}$, respectively, and
  \begin{align} \label{E:PS_ReSet}
    \rho_{\mathcal{X}}(L_\alpha) = \rho_{\mathcal{Y}}(\mathbb{Q}L_\alpha)\cap\rho_{\mathcal{N}}(\mathbb{P}A)
  \end{align}
  for all $\alpha\in\mathbb{R}$. Moreover, for $\zeta\in\rho_{\mathcal{X}}(L_\alpha)$ and $f\in\mathcal{X}$ we have
  \begin{align} \label{E:PS_ReOp}
    \begin{aligned}
      \mathbb{Q}(\zeta-L_\alpha)^{-1}f &= (\zeta-\mathbb{Q}L_\alpha)^{-1}\mathbb{Q}f, \\
      \mathbb{P}(\zeta-L_\alpha)^{-1}f &= (\zeta-\mathbb{P}A)^{-1}\mathbb{P}f-i\alpha(\zeta-\mathbb{P}A)^{-1}\mathbb{P}\Lambda(\zeta-\mathbb{Q}L_\alpha)^{-1}\mathbb{Q}f.
    \end{aligned}
  \end{align}
\end{lemma}

\begin{proof}
  We see that $L_\alpha=A-i\alpha\Lambda$ is closed in $\mathcal{X}$ since $A$ is closed and $\Lambda$ is $A$-compact in $\mathcal{X}$ (see \cite[Theorem IV-1.11]{Kato76}). Similarly, $\mathbb{Q}L_\alpha=\mathbb{Q}A-i\alpha\mathbb{Q}\Lambda$ is closed in $\mathcal{Y}$.

  Let us show \eqref{E:PS_ReSet} and \eqref{E:PS_ReOp}. Since $A$ is self-adjoint and $\mathbb{P}A\subset A\mathbb{P}$ in $\mathcal{X}$, we see that $\mathbb{P}A$ is self-adjoint in $\mathcal{N}$ and thus the residual spectrum of $\mathbb{P}A$ in $\mathcal{N}$ is empty. Hence for each $\zeta\in\sigma_{\mathcal{N}}(\mathbb{P}A)$ there exists a sequence $\{v_k\}_{k=1}^\infty$ in $D_{\mathcal{N}}(\mathbb{P}A)$ such that
  \begin{align} \label{Pf_PRe:PA_EV}
    \|v_k\|_{\mathcal{X}} = 1 \quad\text{for all}\quad k\in\mathbb{N}, \quad \lim_{k\to\infty}\|(\zeta-\mathbb{P}A)v_k\|_{\mathcal{X}} = 0,
  \end{align}
  which includes the case where $\zeta$ is an eigenvalue of $\mathbb{P}A$ with an eigenvector $v_\zeta$ and $v_k=v_\zeta$ for all $k\in\mathbb{N}$. Then since $L_\alpha v=Av=\mathbb{P}Av$ for $v\in D_{\mathcal{N}}(\mathbb{P}A)$, the sequence $\{v_k\}_{k=1}^\infty$ in $\mathcal{X}$ satisfies \eqref{Pf_PRe:PA_EV} with $\mathbb{P}A$ replaced by $L_\alpha$, which means that $\zeta\in\sigma_{\mathcal{X}}(L_\alpha)$. Hence $\sigma_{\mathcal{N}}(\mathbb{P}A) \subset \sigma_{\mathcal{X}}(L_\alpha)$, i.e. $\rho_{\mathcal{X}}(L_\alpha) \subset \rho_{\mathcal{N}}(\mathbb{P}A)$. Let $\zeta\in\rho_{\mathcal{X}}(L_\alpha)\subset\rho_{\mathcal{N}}(\mathbb{P}A)$. If $(\zeta-\mathbb{Q}L_\alpha)u=0$ for $u\in D_{\mathcal{Y}}(\mathbb{Q}L_\alpha)$, then we see by $\mathbb{P}A\subset A\mathbb{P}$ and $\mathbb{P}u=0$ that
  \begin{align*}
    (\zeta-L_\alpha)u = (\zeta-\mathbb{Q}L_\alpha)u-\mathbb{P}L_\alpha u = (\zeta-\mathbb{Q}L_\alpha)u-\mathbb{P}Au+i\alpha\mathbb{P}\Lambda u = i\alpha\mathbb{P}\Lambda u.
  \end{align*}
  Moreover, we can set $v=-i\alpha(\zeta-\mathbb{P}A)^{-1}\mathbb{P}\Lambda u\in D_{\mathcal{N}}(\mathbb{P}A)$ since $\zeta\in\rho_{\mathcal{N}}(\mathbb{P}A)$ and $\mathbb{P}\Lambda u\in\mathcal{N}$. Then we observe by $L_\alpha v=Av=\mathbb{P}Av$ that
  \begin{align*}
    (\zeta-L_\alpha)(u+v) = (\zeta-L_\alpha)u+(\zeta-\mathbb{P}A)v = i\alpha\mathbb{P}\Lambda u-i\alpha\mathbb{P}\Lambda u = 0,
  \end{align*}
  which yields $u+v=0$ by $\zeta\in\rho_{\mathcal{X}}(L_\alpha)$. Hence $u=\mathbb{Q}(u+v)=0$ and $\zeta-\mathbb{Q}L_\alpha$ is injective. Also, for $f\in\mathcal{Y}\subset\mathcal{X}$ let $w=(\zeta-L_\alpha)^{-1}f \in D_{\mathcal{X}}(L_\alpha)$. Then since
  \begin{align*}
    f = (\zeta-L_\alpha)u+(\zeta-A)v, \quad u = \mathbb{Q}w \in D_{\mathcal{Y}}(\mathbb{Q}L_\alpha), \quad v = \mathbb{P}w \in D_{\mathcal{N}}(\mathbb{P}A)
  \end{align*}
  by $f=(\zeta-L_\alpha)w$ and $L_\alpha v=Av$, we see by $\mathbb{Q}u=u$, $\mathbb{Q}v=0$, and $\mathbb{Q}A\subset A\mathbb{Q}$ that
  \begin{align*}
    f = \mathbb{Q}f = (\zeta-\mathbb{Q}L_\alpha)u+(\zeta-A)\mathbb{Q}v = (\zeta-\mathbb{Q}L_\alpha)u.
  \end{align*}
  Hence $\zeta-\mathbb{Q}L_\alpha$ is surjective and we get $\zeta\in\rho_{\mathcal{Y}}(\mathbb{Q}L_\alpha)$, i.e. $\rho_{\mathcal{X}}(L_\alpha)\subset\rho_{\mathcal{Y}}(\mathbb{Q}L_\alpha)\cap\rho_{\mathcal{N}}(\mathbb{P}A)$. Conversely, let $\zeta\in\rho_{\mathcal{Y}}(\mathbb{Q}L_\alpha)\cap \rho_{\mathcal{N}}(\mathbb{P}A)$. If $(\zeta-L_\alpha)w=0$ for $w\in D_{\mathcal{X}}(L_\alpha)$, then
  \begin{align} \label{Pf_PRe:L_inj}
    (\zeta-L_\alpha)u+(\zeta-A)v = 0, \quad u = \mathbb{Q}w \in D_{\mathcal{Y}}(\mathbb{Q}L_\alpha), \quad v = \mathbb{P}w \in D_{\mathcal{N}}(\mathbb{P}A)
  \end{align}
  by $L_\alpha v=Av$. We apply $\mathbb{Q}$ to \eqref{Pf_PRe:L_inj} and use $\mathbb{Q}u=u$, $\mathbb{Q}v=0$, and $\mathbb{Q}A\subset A\mathbb{Q}$ to find that $(\zeta-\mathbb{Q}L_\alpha)u=0$. Thus $u=0$ by $\zeta\in\rho_{\mathcal{Y}}(\mathbb{Q}L_\alpha)$. Then we also have $(\zeta-\mathbb{P}A)v=0$ by \eqref{Pf_PRe:L_inj}, $\mathbb{P}v=v$, and $\mathbb{P}A\subset A\mathbb{P}$, which yields $v=0$ since $\zeta\in\rho_{\mathcal{N}}(\mathbb{P}A)$. Hence $w=u+v=0$ and $\zeta-L_\alpha$ is injective. Also, for $f\in\mathcal{X}$ let $w=u+v_1+v_2$ with
  \begin{align} \label{Pf_PRe:Def_uv}
    u = (\zeta-\mathbb{Q}L_\alpha)^{-1}\mathbb{Q}f, \quad v_1 = (\zeta-\mathbb{P}A)^{-1}\mathbb{P}f, \quad v_2 = -i\alpha(\zeta-\mathbb{P}A)^{-1}\mathbb{P}\Lambda u.
  \end{align}
  Then since $u\in D_{\mathcal{Y}}(\mathbb{Q}L_\alpha)$ and $v_1,v_2\in D_{\mathcal{N}}(\mathbb{P}A)$, we have $w\in D_{\mathcal{X}}(L_\alpha)$ and
  \begin{align} \label{Pf_PRe:Cal_w}
    \begin{aligned}
      (\zeta-L_\alpha)w &= (\zeta-L_\alpha)u+(\zeta-L_\alpha)(v_1+v_2) \\
      &= (\zeta-\mathbb{Q}L_\alpha)u+i\alpha\mathbb{P}\Lambda u+(\zeta-\mathbb{P}A)(v_1+v_2) = f
    \end{aligned}
  \end{align}
  by $\mathbb{Q}u=u$, $\mathbb{P}v_j=v_j$ for $j=1,2$, $\mathbb{Q}A\subset A\mathbb{Q}$, and $\mathbb{P}A\subset A\mathbb{P}$. Thus $\zeta-L_\alpha$ is surjective, i.e. $\zeta\in\rho_{\mathcal{X}}(L_\alpha)$, and we obtain \eqref{E:PS_ReSet}. Also, when $\zeta\in\rho_{\mathcal{X}}(L_\alpha)$, we see by \eqref{Pf_PRe:Cal_w} that
  \begin{align*}
    \mathbb{Q}(\zeta-L_\alpha)^{-1}f = \mathbb{Q}w = u, \quad \mathbb{P}(\zeta-L_\alpha)^{-1}f = \mathbb{P}w = v_1+v_2, \quad f\in\mathcal{X}
  \end{align*}
  for $w=u+v_1+v_2$ with $u$, $v_1$, and $v_2$ given by \eqref{Pf_PRe:Def_uv}. Hence \eqref{E:PS_ReOp} follows.
\end{proof}

\begin{lemma} \label{L:PS_QP}
  Under Assumptions \ref{As:A}--\ref{As:La_02}, we have
  \begin{align} \label{E:QB2}
    \mathbb{Q}B_2 = B_2 \quad\text{on}\quad \mathcal{X}, \quad \mathrm{Im}(\Lambda u,\mathbb{Q}B_2u)_{\mathcal{X}} = \mathrm{Im}(\Lambda u,B_2u)_{\mathcal{X}} = 0, \quad u\in D_{\mathcal{X}}(\Lambda).
  \end{align}
\end{lemma}

\begin{proof}
  For $u,v\in\mathcal{X}$ we have $(\mathbb{P}B_2u,v)_{\mathcal{X}}=(u,B_2\mathbb{P}v)_{\mathcal{X}}=0$ since $B_2$ is self-adjoint and $\mathcal{N}=N_{\mathcal{X}}(B_2)$ in $\mathcal{X}$. Hence $\mathbb{P}B_2=0$ and $\mathbb{Q}B_2=B_2$ on $\mathcal{X}$. Also,
  \begin{align*}
    \mathrm{Im}(\Lambda u,\mathbb{Q}B_2u)_{\mathcal{X}} = \mathrm{Im}(B_1B_2u,B_2u)_{\mathcal{X}} = \mathrm{Im}(B_1B_2u,B_2u)_{\mathcal{H}} = 0
  \end{align*}
  for $u\in D_{\mathcal{X}}(\Lambda)$ by Assumption \ref{As:La_02} (i) and (ii), where the last equality holds since $B_1$ is symmetric in $\mathcal{H}$ and $B_2u\in D_{\mathcal{H}}(B_1)$. Thus the second relation of \eqref{E:QB2} is valid.
\end{proof}

\begin{lemma} \label{L:Res_Half}
  Under Assumptions \ref{As:A}--\ref{As:La_02}, for all $\alpha\in\mathbb{R}$ we have
  \begin{align} \label{E:Res_Half}
    \{\zeta\in\mathbb{C} \mid \mathrm{Re}\,\zeta \geq 0\} \subset \rho_{\mathcal{X}}(L_\alpha) \subset \rho_{\mathcal{Y}}(\mathbb{Q}L_\alpha).
  \end{align}
\end{lemma}

\begin{proof}
  It suffices to verify the first inclusion since the second one follows from \eqref{E:PS_ReSet}. Since $A$ is self-adjoint in $\mathcal{X}$ and satisfies \eqref{E:Ab_A_Po}, and since $\Lambda$ is $A$-compact,
  \begin{align} \label{Pf_RH:TSig_L}
    \tilde{\sigma}_{\mathcal{X}}(L_\alpha) = \tilde{\sigma}_{\mathcal{X}}(A-i\alpha\Lambda) = \tilde{\sigma}_{\mathcal{X}}(A) \subset \sigma_{\mathcal{X}}(A) \subset (-\infty,-C_A].
  \end{align}
  Let $\zeta\in\sigma_{\mathcal{X}}(L_\alpha)\setminus\tilde{\sigma}_{\mathcal{X}}(L_\alpha)$. Then $\zeta$ is an eigenvalue of $L_\alpha$. Let $w\in D_{\mathcal{X}}(L_\alpha)$, $w\neq0$ be a corresponding eigenvector. If $w\in\mathcal{N}$, then $(\zeta-A)w=(\zeta-L_\alpha)w=0$ and $\zeta$ is an eigenvalue of $A$. Hence $\zeta\in(-\infty,-C_A]$, since $A$ is self-adjoint in $\mathcal{X}$ and satisfies \eqref{E:Ab_A_Po}. Suppose that $w\not\in\mathcal{N}$, i.e. $u=\mathbb{Q}w\in D_{\mathcal{Y}}(\mathbb{Q}L_\alpha)$ satisfies $u\neq0$. Then we apply $\mathbb{Q}$ to
  \begin{align*}
    0 = (\zeta-L_\alpha)w = (\zeta-L_\alpha)u+(\zeta-A)v, \quad v = \mathbb{P}w \in D_{\mathcal{N}}(\mathbb{P}A)
  \end{align*}
  and use $\mathbb{Q}u=u$, $\mathbb{Q}v=0$, and $\mathbb{Q}A\subset A\mathbb{Q}$ to get $(\zeta-\mathbb{Q}L_\alpha)u=0$, i.e. $\zeta u=\mathbb{Q}L_\alpha u$. Thus
  \begin{align*}
    \zeta(u,B_2u)_{\mathcal{X}} = (\mathbb{Q}L_\alpha u,B_2u)_{\mathcal{X}} = (L_\alpha u,\mathbb{Q}B_2u)_{\mathcal{X}} = (Au,\mathbb{Q}B_2u)_{\mathcal{X}}-i\alpha(\Lambda u,\mathbb{Q}B_2u)_{\mathcal{X}}.
  \end{align*}
  Noting that $(u,B_2u)_{\mathcal{X}}$ is real and positive by \eqref{E:uB2u} and $u\neq0$, we take the real part of the above equality and apply \eqref{E:Ab_A_Po}, \eqref{E:AB2u}, and \eqref{E:QB2}. Then
  \begin{align*}
    (\mathrm{Re}\,\zeta)(u,B_2u)_{\mathcal{X}} &= \mathrm{Re}(Au,B_2u)_{\mathcal{X}} = -\mathrm{Re}(-Au,B_2u)_{\mathcal{X}} \\
    &\leq -C\|(-A)^{1/2}u\|_{\mathcal{X}}^2= -C(-Au,u)_{\mathcal{X}} \leq -C_AC\|u\|_{\mathcal{X}}^2
  \end{align*}
  and we divide both sides by $(u,B_2u)_{\mathcal{X}}>0$ to find that
  \begin{align*}
    \mathrm{Re}\,\zeta \leq -\frac{C_AC\|u\|_{\mathcal{X}}^2}{(u,B_2u)_{\mathcal{X}}} < 0.
  \end{align*}
  Hence $\sigma_{\mathcal{X}}(L_\alpha)\setminus\tilde{\sigma}_{\mathcal{X}}(L_\alpha)$ is contained in $\{\zeta\in\mathbb{C} \mid \mathrm{Re}\,\zeta<0\}$, and we conclude by this fact and \eqref{Pf_RH:TSig_L} that $\sigma_{\mathcal{X}}(L_\alpha)\subset\{\zeta\in\mathbb{C}\mid\mathrm{Re}\,\zeta<0\}$, i.e. the first inclusion of \eqref{E:Res_Half} is valid.
\end{proof}

Now let us give the estimate for the semigroup generated by $L_\alpha=A-i\alpha\Lambda$.

\begin{theorem} \label{T:Semi_BoAl}
  Under Assumptions \ref{As:A}--\ref{As:La_02}, the operator $L_\alpha$ generates an analytic semigroup $\{e^{tL_\alpha}\}_{t\geq0}$ in $\mathcal{X}$ for all $\alpha\in\mathbb{R}$. Moreover, there exist positive constants $C_1$ and $C_2$ depending only on $\|B_2\|_{\mathcal{X}\to\mathcal{X}}$ and the constants appearing in \eqref{E:uB2u} and \eqref{E:AB2u} (and in particular independent of the constant $C_A$ appearing in \eqref{E:Ab_A_Po}) such that
  \begin{align} \label{E:Semi_BoAl}
    \|\mathbb{Q}e^{tL_\alpha}f\|_{\mathcal{X}} \leq C_1e^{-C_2t/\Phi_{\mathcal{Y}}(-\mathbb{Q}L_\alpha)}\|\mathbb{Q}f\|_{\mathcal{X}}, \quad t\geq0, \, f\in\mathcal{X}
  \end{align}
  for all $\alpha\in\mathbb{R}$, where $\Phi_{\mathcal{Y}}(-\mathbb{Q}L_\alpha)=\sup_{\lambda\in\mathbb{R}}\|(i\lambda-\mathbb{Q}L_\alpha)^{-1}\|_{\mathcal{Y}\to\mathcal{Y}}$.
\end{theorem}

The proof of \eqref{E:Semi_BoAl} relies on the following Gearhart--Pr\"{u}ss type theorem shown by Wei \cite{Wei21}. A closed operator $S$ on a Hilbert space $(\mathcal{H},(\cdot,\cdot)_\mathcal{H})$ is called $m$-accretive if
\begin{align*}
  \{\zeta\in\mathbb{C} \mid \mathrm{Re}\,\zeta<0\} \subset \rho_{\mathcal{H}}(S), \quad \mathrm{Re}(Su,u)_{\mathcal{H}} \geq 0, \quad u\in D_{\mathcal{H}}(S).
\end{align*}
An $m$-accretive operator $S$ is densely defined and satisfies (see \cite[Section V-3.10]{Kato76})
\begin{align*}
  \{\zeta\in\mathbb{C} \mid \mathrm{Re}\,\zeta>0\} \subset \rho_{\mathcal{H}}(-S), \quad \|(\zeta+S)^{-1}\|_{\mathcal{H}\to\mathcal{H}} \leq \frac{1}{\mathrm{Re}\,\zeta}, \quad \mathrm{Re}\,\zeta>0.
\end{align*}
Thus $-S$ generates a contraction semigroup $\{e^{-tS}\}_{t\geq0}$ in $\mathcal{H}$ by the Hille--Yosida theorem.

\begin{theorem}[{\cite[Theorem 1.3]{Wei21}}] \label{T:GP_W}
  Let $S$ be an $m$-accretive operator on $\mathcal{H}$. Then
  \begin{align*}
    \|e^{-tS}\|_{\mathcal{H}\to\mathcal{H}} \leq e^{-t\Psi_{\mathcal{H}}(S)+\pi/2}, \quad t\geq0,
  \end{align*}
  where $\Psi_{\mathcal{H}}(S)$ is the pseudospectral bound of $S$ in $\mathcal{H}$ given by
  \begin{align*}
    \Psi_{\mathcal{H}}(S) = \inf\{\, \|(i\lambda+S)f\|_{\mathcal{H}} \mid \lambda\in\mathbb{R},\, f\in D_{\mathcal{H}}(S), \, \|f\|_{\mathcal{H}}=1 \}.
  \end{align*}
\end{theorem}

\begin{proof}[Proof of Theorem \ref{T:Semi_BoAl}]
  Fix $\delta\in(\pi/2,\pi)$ and let $\Sigma=\{\zeta\in\mathbb{C} \mid |\arg\zeta|<\delta, \, \zeta\neq0\}$. Since $A$ is self-adjoint in $\mathcal{X}$ and satisfies \eqref{E:Ab_A_Po}, there exists a constant $C>0$ such that
  \begin{align*}
    \Sigma \subset \rho_{\mathcal{X}}(A), \quad \|(\zeta-A)^{-1}\|_{\mathcal{X}\to\mathcal{X}} \leq \frac{C}{|\zeta|}, \quad \zeta \in \Sigma.
  \end{align*}
  Also, for each $\alpha\in\mathbb{R}$, since $i\alpha\Lambda$ is closed and $A$-compact in $\mathcal{X}$, it is $A$-bounded with $A$-bound zero (see \cite[Lemma III.2.16]{EngNag00}). Thus, by \cite[Lemma III.2.6]{EngNag00}, there exist constants $r_\alpha,C_\alpha>0$ depending on $\alpha$ such that $\Sigma\cap\{\zeta\in\mathbb{C}\mid|\zeta|>r_\alpha\}\subset\rho_{\mathcal{X}}(L_\alpha)$ and
  \begin{align*}
    \|(\zeta-L_\alpha)^{-1}\|_{\mathcal{X}\to\mathcal{X}} \leq \frac{C_\alpha}{|\zeta|}, \quad \zeta \in \Sigma\cap\{\zeta\in\mathbb{C} \mid |\zeta| > r_\alpha\}.
  \end{align*}
  By this fact, we see that $L_\alpha-\gamma_\alpha$ is sectorial for a sufficiently large $\gamma_\alpha>0$, i.e.
  \begin{align} \label{Pf_SBA:Re_La}
    \begin{gathered}
      \Sigma_\alpha+\gamma_\alpha \subset \rho_{\mathcal{X}}(L_\alpha), \quad \Sigma_\alpha = \{\zeta\in\mathbb{C} \mid |\arg\zeta|<\delta_\alpha, \, \zeta\neq0\}, \\
      \|(\zeta+\gamma_\alpha-L_\alpha)^{-1}\|_{\mathcal{X}\to\mathcal{X}} \leq \frac{C'_\alpha}{|\zeta|}, \quad \zeta\in\Sigma_\alpha
    \end{gathered}
  \end{align}
  with some angle $\delta_\alpha\in(\pi/2,\pi)$ and constant $C'_\alpha>0$ depending on $\alpha$. Hence $L_\alpha$ generates an analytic semigroup $\{e^{tL_\alpha}\}_{t\geq0}$ in $\mathcal{X}$ represented by the Dunford integral
  \begin{align} \label{Pf_SBA:Semi_La}
    e^{tL_\alpha}f = \frac{1}{2\pi i}\int_{\Gamma_\alpha}e^{t(\zeta+\gamma_\alpha)}(\zeta+\gamma_\alpha-L_\alpha)^{-1}f\,d\zeta, \quad t>0, \, f\in\mathcal{X},
  \end{align}
  where $\Gamma_\alpha$ is a piecewise smooth curve in $\Sigma_\alpha$ going from $\infty e^{-i\sigma}$ to $\infty e^{i\sigma}$ with $\sigma\in(0,\delta_\alpha)$. Also, since $\rho_{\mathcal{X}}(L_\alpha)\subset\rho_{\mathcal{Y}}(\mathbb{Q}L_\alpha)$ by \eqref{E:PS_ReSet} and
  \begin{align*}
    \|(\zeta-\mathbb{Q}L_\alpha)^{-1}\|_{\mathcal{Y}\to\mathcal{Y}} = \|\mathbb{Q}(\zeta-L_\alpha)^{-1}\|_{\mathcal{X}\to\mathcal{X}} \leq \|(\zeta-L_\alpha)^{-1}\|_{\mathcal{X}\to\mathcal{X}}, \quad \zeta\in\rho_{\mathcal{X}}(L_\alpha)
  \end{align*}
  by the first equality of \eqref{E:PS_ReOp}, it follows from \eqref{Pf_SBA:Re_La} that
  \begin{align*}
    \Sigma_\alpha+\gamma_\alpha \subset \rho_{\mathcal{Y}}(\mathbb{Q}L_\alpha), \quad \|(\zeta+\gamma_\alpha-\mathbb{Q}L_\alpha)^{-1}\|_{\mathcal{Y}\to\mathcal{Y}} \leq \frac{C'_\alpha}{|\zeta|}, \quad \zeta\in\Sigma_\alpha.
  \end{align*}
  Hence we can define the analytic semigroup generated by $\mathbb{Q}L_\alpha$ in $\mathcal{Y}$ by
  \begin{align} \label{Pf_SBA:Semi_QL}
    e^{t\mathbb{Q}L_\alpha}g = \frac{1}{2\pi i}\int_{\Gamma_\alpha}e^{t(\zeta+\gamma_\alpha)}(\zeta+\gamma_\alpha-\mathbb{Q}L_\alpha)^{-1}g\,d\zeta, \quad t>0, \, g\in\mathcal{Y},
  \end{align}
  where $\Gamma_\alpha$ is the same curve as in \eqref{Pf_SBA:Semi_La}. By \eqref{E:PS_ReOp}, \eqref{Pf_SBA:Semi_La}, and \eqref{Pf_SBA:Semi_QL}, we get
  \begin{align} \label{Pf_SBA:QetL}
    \mathbb{Q}e^{tL_\alpha}f = e^{t\mathbb{Q}L_\alpha}\mathbb{Q}f, \quad t\geq0, \, f\in\mathcal{X},
  \end{align}
  where we included the trivial case $t=0$. Thus, to prove \eqref{E:Semi_BoAl}, it suffices to estimate the right-hand side in $\mathcal{Y}$. In what follows, we write $C$ and $C'$ for general positive constants depending only on $\|B_2\|_{\mathcal{X}\to\mathcal{X}}$ and the constants appearing in \eqref{E:uB2u} and \eqref{E:AB2u}.

  For $u,v\in\mathcal{Y}$ let $(u,v)_{\mathcal{Y}'}=(u,B_2v)_{\mathcal{X}}$. Then since $B_2$ is bounded and self-adjoint in $\mathcal{X}$ and satisfies \eqref{E:uB2u}, we have
  \begin{align} \label{Pf_SBA:Equiv}
    (u,u)_{\mathcal{Y}'}\in\mathbb{R}, \quad C\|u\|_{\mathcal{X}}^2 \leq (u,u)_{\mathcal{Y}'} \leq C'\|u\|_{\mathcal{X}}^2, \quad u\in\mathcal{Y}.
  \end{align}
  Hence $(\cdot,\cdot)_{\mathcal{Y}'}$ defines an inner product on $\mathcal{Y}$ equivalent to $(\cdot,\cdot)_{\mathcal{X}}$ restricted on $\mathcal{Y}$. We write $\mathcal{Y}'$ for the Hilbert space $\mathcal{Y}$ equipped with inner product $(\cdot,\cdot)_{\mathcal{Y}'}$. Then
  \begin{align} \label{Pf_SBA:Set_QL}
    \{\zeta\in\mathbb{C} \mid \mathrm{Re}\,\zeta \leq 0\} \subset \rho_{\mathcal{Y}}(-\mathbb{Q}L_\alpha) = \rho_{\mathcal{Y}'}(-\mathbb{Q}L_\alpha)
  \end{align}
  by \eqref{E:Res_Half}. Moreover, for $u\in D_{\mathcal{Y}'}(\mathbb{Q}L_\alpha)=D_{\mathcal{Y}}(\mathbb{Q}L_\alpha)=D_{\mathcal{X}}(A)\cap\mathcal{Y}$, since
  \begin{align*}
    (-\mathbb{Q}L_\alpha u,u)_{\mathcal{Y}'} &= (-\mathbb{Q}Au,B_2u)_{\mathcal{X}}+i\alpha(\mathbb{Q}\Lambda u,B_2u)_{\mathcal{X}} \\
    &= (-Au,\mathbb{Q}B_2u)_{\mathcal{X}}+i\alpha(\Lambda u,\mathbb{Q}B_2u)_{\mathcal{X}},
  \end{align*}
  we take the real part of this equality and apply \eqref{E:AB2u} and \eqref{E:QB2} to find that
  \begin{align*}
    \mathrm{Re}(-\mathbb{Q}L_\alpha u,u)_{\mathcal{Y}'} = \mathrm{Re}(-Au,B_2u)_{\mathcal{X}} \geq C\|(-A)^{1/2}u\|_{\mathcal{X}}^2 \geq 0, \quad u\in D_{\mathcal{Y}'}(\mathbb{Q}L_\alpha).
  \end{align*}
  Hence $-\mathbb{Q}L_\alpha$ is $m$-accretive in $\mathcal{Y}'$ and we can use Theorem \ref{T:GP_W} to get
  \begin{align} \label{Pf_SBA:Semi_Yp}
    \|e^{t\mathbb{Q}L_\alpha}\|_{\mathcal{Y}'\to\mathcal{Y}'} \leq e^{-t\Psi_{\mathcal{Y}'}(-\mathbb{Q}L_\alpha)+\pi/2}, \quad t\geq0,
  \end{align}
  where $\Psi_{\mathcal{Y}'}(-\mathbb{Q}L_\alpha)$ is the pseudospectral bound of $-\mathbb{Q}L_\alpha$ in $\mathcal{Y}'$ given by
  \begin{align*}
    \Psi_{\mathcal{Y}'}(-\mathbb{Q}L_\alpha) = \inf\{\|(i\lambda-\mathbb{Q}L_\alpha)g\|_{\mathcal{Y}'} \mid \lambda\in\mathbb{R}, \, g\in D_{\mathcal{Y}'}(\mathbb{Q}L_\alpha), \, \|g\|_{\mathcal{Y}'} = 1\}.
  \end{align*}
  Since $i\mathbb{R}\subset\rho_{\mathcal{Y}'}(-\mathbb{Q}L_\alpha)$ by \eqref{Pf_SBA:Set_QL}, we easily find that
  \begin{align*}
    \Psi_{\mathcal{Y}'}(-\mathbb{Q}L_\alpha) \geq \left(\sup_{\lambda\in\mathbb{R}}\|(i\lambda-\mathbb{Q}L_\alpha)^{-1}\|_{\mathcal{Y}'\to\mathcal{Y}'}\right)^{-1}.
  \end{align*}
  Moreover, we observe by \eqref{Pf_SBA:Equiv} that
  \begin{align*}
    \sup_{\lambda\in\mathbb{R}}\|(i\lambda-\mathbb{Q}L_\alpha)^{-1}\|_{\mathcal{Y}'\to\mathcal{Y}'} \leq C\sup_{\lambda\in\mathbb{R}}\|(i\lambda-\mathbb{Q}L_\alpha)^{-1}\|_{\mathcal{Y}\to\mathcal{Y}} = C\Phi_{\mathcal{Y}}(-\mathbb{Q}L_\alpha).
  \end{align*}
  Thus $\Psi_{\mathcal{Y}'}(-\mathbb{Q}L_\alpha)\geq C/\Phi_{\mathcal{Y}}(-\mathbb{Q}L_\alpha)>0$. By this inequality, \eqref{Pf_SBA:Equiv}, and \eqref{Pf_SBA:Semi_Yp},
  \begin{align*}
    \|e^{t\mathbb{Q}L_\alpha}g\|_{\mathcal{X}} &\leq C\|e^{t\mathbb{Q}L_\alpha}g\|_{\mathcal{Y}'} \leq Ce^{-t\Psi_{\mathcal{Y}'}(-\mathbb{Q}L_\alpha)}\|g\|_{\mathcal{Y}'} \leq Ce^{-Ct/\Phi_{\mathcal{Y}}(-\mathbb{Q}L_\alpha)}\|g\|_{\mathcal{X}}
  \end{align*}
  for all $t\geq0$ and $g\in\mathcal{Y}$, and we obtain \eqref{E:Semi_BoAl} by this inequality and \eqref{Pf_SBA:QetL}.
\end{proof}

Under additional assumptions, we can also show that $\Phi_{\mathcal{Y}}(-\mathbb{Q}L_\alpha)$ converges to zero as $|\alpha|\to\infty$. Recall that we write $N_{\mathcal{X}}(\Lambda)=\mathcal{N}$ and $R_{\mathcal{X}}(\Lambda)$ for the kernel and range of $\Lambda$ in $\mathcal{X}$. Also, let $\Lambda^\ast$ be the adjoint of $\Lambda$ in $\mathcal{X}$.

\begin{theorem} \label{T:Phi_Conv}
  Under Assumptions \ref{As:A}--\ref{As:La_02}, suppose further that
  \begin{itemize}
    \item[\textrm{(a)}] $A$ has a compact resolvent in $\mathcal{X}$,
    \item[\textrm{(b)}] there exists a constant $C>0$ such that
    \begin{align*}
      |(\Lambda u,u)_{\mathcal{X}}| \leq C(-Au,u)_{\mathcal{X}}, \quad u \in D_{\mathcal{X}}(A),
    \end{align*}
    \item[\textrm{(c)}] $D_{\mathcal{X}}(A)\subset D_{\mathcal{X}}(\Lambda^\ast)$ and $N_{\mathcal{X}}(\Lambda)\subset D_{\mathcal{X}}(\Lambda^\ast)$,
    \item[\textrm{(d)}] $N_{\mathcal{X}}(\Lambda)\cap R_{\mathcal{X}}(\Lambda)=\{0\}$, and
    \item[\textrm{(e)}] $\Lambda$ does not have eigenvalues in $\mathbb{R}\setminus\{0\}$.
  \end{itemize}
  Then we have
  \begin{align} \label{E:Phi_Conv}
    \lim_{|\alpha|\to\infty}\Phi_{\mathcal{Y}}(-\mathbb{Q}L_\alpha) = \lim_{|\alpha|\to\infty}\sup_{\lambda\in\mathbb{R}}\|(i\lambda-\mathbb{Q}L_\alpha)^{-1}\|_{\mathcal{Y}\to\mathcal{Y}} = 0
  \end{align}
  and $\lim_{|\alpha|\to\infty}\sup_{t\geq\tau}\|\mathbb{Q}e^{tL_\alpha}\|_{\mathcal{X}\to\mathcal{X}}=0$ for each $\tau>0$.
\end{theorem}

\begin{proof}
  The second statement follows from \eqref{E:Phi_Conv} since
  \begin{align*}
    \sup_{t\geq\tau}\|\mathbb{Q}e^{tL_\alpha}\|_{\mathcal{X}\to\mathcal{X}}\leq C_1e^{-C_2\tau/\Phi_{\mathcal{Y}}(-\mathbb{Q}L_\alpha)}, \quad \tau>0
  \end{align*}
  by \eqref{E:Semi_BoAl}. The proof of \eqref{E:Phi_Conv} is the same as Step 3 of the proof of \cite[Theorem 2.4]{IbMaMa19}, but here we give it for the completeness since the notations and assumptions given in this paper are slightly different from those given in \cite{IbMaMa19}.

  Assume to the contrary that \eqref{E:Phi_Conv} does not hold, i.e. there exist a constant $\delta>0$ and sequences $\{\alpha_n\}_{n=1}^\infty$ and $\{\lambda_n\}_{n=1}^\infty$ in $\mathbb{R}$ and $\{f_n\}_{n=1}^\infty$ in $\mathcal{Y}$ such that
  \begin{align*}
    \lim_{n\to\infty}|\alpha_n| = \infty, \quad \|f_n\|_{\mathcal{X}} = 1, \quad \|(i\lambda_n-\mathbb{Q}L_{\alpha_n})^{-1}f_n\|_{\mathcal{X}} \geq \delta \quad\text{for all}\quad n\in\mathbb{N}.
  \end{align*}
  Let $u_n=(i\lambda_n-\mathbb{Q}L_{\alpha_n})^{-1}f_n \in D_{\mathcal{Y}}(\mathbb{Q}L_{\alpha_n})=D_{\mathcal{Y}}(\mathbb{Q}A)$. Then $\|u_n\|_{\mathcal{X}}\geq\delta$ and
  \begin{align} \label{Pf_PC:un_Eq}
    i\lambda_nu_n-Au_n+i\alpha_n\mathbb{Q}\Lambda u_n = f_n
  \end{align}
  by $(i\lambda_n-\mathbb{Q}L_{\alpha_n})u_n=f_n$ and $\mathbb{Q}Au_n=A\mathbb{Q}u_n=Au_n$, and thus
  \begin{align*}
    i\lambda_n(u_n,B_2u_n)_{\mathcal{X}}+(-Au_n,B_2u_n)_{\mathcal{X}}+i\alpha_n(\mathbb{Q}\Lambda u_n,B_2u_n)_{\mathcal{X}} = (f_n,B_2u_n)_{\mathcal{X}}.
  \end{align*}
  We take the real part of this equality and use
  \begin{align*}
    \mathrm{Im}(u_n,B_2u_n)_{\mathcal{X}} = 0, \quad \mathrm{Im}(\mathbb{Q}\Lambda u_n,B_2u_n)_{\mathcal{X}} = \mathrm{Im}(\Lambda u_n,\mathbb{Q}B_2u_n)_{\mathcal{X}} = 0
  \end{align*}
  by the self-adjointness of $B_2$ in $\mathcal{X}$ and \eqref{E:QB2} to get $\mathrm{Re}(-Au_n,B_2u_n)_{\mathcal{X}}=\mathrm{Re}(f_n,B_2u_n)$. By this equality, \eqref{E:Ab_A_Po}, \eqref{E:AB2u}, the boundedness of $B_2$ in $\mathcal{X}$, and $\|f_n\|_{\mathcal{X}}=1$, we have
  \begin{align*}
    \|(-A)^{1/2}u_n\|_{\mathcal{X}}^2 &\leq C\,\mathrm{Re}(-Au_n,B_2u_n)_{\mathcal{X}} \leq C\|f_n\|_{\mathcal{X}}\|B_2u_n\|_{\mathcal{X}} \\
    &\leq C\|u_n\|_{\mathcal{X}} \leq C(-Au_n,u_n)_{\mathcal{X}}^{1/2} = C\|(-A)^{1/2}u_n\|_{\mathcal{X}}
  \end{align*}
  and thus, by \eqref{E:Ab_A_Po} and the above inequality,
  \begin{align} \label{Pf_PC:un_Bd}
    \|u_n\|_{\mathcal{X}} \leq C(-Au_n,u_n)_{\mathcal{X}}^{1/2} = C\|(-A)^{1/2}u_n\|_{\mathcal{X}} \leq C.
  \end{align}
  Here and in what follows, $C$ denotes a general positive constant independent of $n$. Now we observe that $(-A)^{1/2}$ is strictly positive, self-adjoint, and with compact resolvent in $\mathcal{X}$ since $-A$ has the same properties by Assumption \ref{As:A} and the condition (a) (see \cite[Theorems V.3.35 and V.3.49]{Kato76}). Thus $(-A)^{-1/2}$ exists and is compact in $\mathcal{X}$. Moreover, since $u_n=(-A)^{-1/2}(-A)^{1/2}u_n$ and $\{(-A)^{1/2}u_n\}_{n=1}^\infty$ is bounded in $\mathcal{X}$ by \eqref{Pf_PC:un_Bd}, we see that $\{u_n\}_{n=1}^\infty$ converges (up to a subsequence) to some $u_\infty$ strongly in $\mathcal{X}$. Then $u_\infty\in\mathcal{Y}$ since $\mathcal{Y}$ is closed in $\mathcal{X}$ and $u_n\in\mathcal{Y}$ for all $n\in\mathbb{N}$. Also, since $\|u_n\|_{\mathcal{X}}\geq\delta$ for all $n\in\mathbb{N}$,
  \begin{align} \label{Pf_PC:uin_del}
    \|u_\infty\|_{\mathcal{X}} = \lim_{n\to\infty}\|u_n\|_{\mathcal{X}} \geq \delta.
  \end{align}
  Let us show $u_\infty=0$. Since $|\alpha_n|\to\infty$ as $n\to\infty$, we may assume $|\alpha_n|\geq1$ for all $n\in\mathbb{N}$. We set $\mu_n=\lambda_n/\alpha_n\in\mathbb{R}$ and divide both sides of \eqref{Pf_PC:un_Eq} by $\alpha_n$ to get
  \begin{align} \label{Pf_PC:un_mu}
    i\mu_nu_n-\frac{1}{\alpha_n}Au_n+i\mathbb{Q}\Lambda u_n = \frac{1}{\alpha_n}f_n.
  \end{align}
  Then, taking the imaginary part of the inner product of \eqref{Pf_PC:un_mu} with $u_n$ and noting that $A$ is self-adjoint in $\mathcal{X}$ and $\mathbb{Q}u_n=u_n$, we find that
  \begin{align*}
    \mu_n\|u_n\|_{\mathcal{X}}^2+\mathrm{Re}(\Lambda u_n,u_n)_{\mathcal{X}} = \frac{1}{\alpha_n}\mathrm{Im}(f_n,u_n)_{\mathcal{X}}
  \end{align*}
  and thus, by the condition (b), $\|f_n\|_{\mathcal{X}}=1$, \eqref{Pf_PC:un_Bd}, and $|\alpha_n|\geq1$,
  \begin{align*}
    |\mu_n| \, \|u_n\|_{\mathcal{X}}^2 &\leq |(\Lambda u_n,u_n)_{\mathcal{X}}|+|(f_n,u_n)_{\mathcal{X}}| \leq C(-Au_n,u_n)_{\mathcal{X}}+\|f_n\|_{\mathcal{X}}\|u_n\|_{\mathcal{X}} \leq C.
  \end{align*}
  Since $\|u_n\|_{\mathcal{X}}\geq\delta$ for all $n\in\mathbb{N}$, we see by the above inequality that $\{\mu_n\}_{n=1}^\infty$ is bounded and thus converges (up to a subsequence) to some $\mu_\infty\in\mathbb{R}$. Now we show that
  \begin{align} \label{Pf_PC:uin_Lam}
    u_\infty \in D_{\mathcal{X}}(\Lambda), \quad \Lambda u_\infty = -\mu_\infty u_\infty+\mathbb{P}\Lambda u_\infty,
  \end{align}
  where $\mathbb{P}=I-\mathbb{Q}$ is the orthogonal projection from $\mathcal{X}$ onto $N_{\mathcal{X}}(\Lambda)$. Let $v\in D_{\mathcal{X}}(A)$. Then $v,\mathbb{P}v\in D_{\mathcal{X}}(\Lambda^\ast)$ by the condition (c) and thus $\mathbb{Q}v=v-\mathbb{P}v\in D_{\mathcal{X}}(\Lambda^\ast)$. We take the inner product of \eqref{Pf_PC:un_mu} with $v$ and use the self-adjointness of $A$ in $\mathcal{X}$ to get
  \begin{align*}
    i\mu_n(u_n,v)_{\mathcal{X}}-\frac{1}{\alpha_n}(u_n,Av)_{\mathcal{X}}+i(u_n,\Lambda^\ast\mathbb{Q}v)_{\mathcal{X}} = \frac{1}{\alpha_n}(f_n,u_n)_{\mathcal{X}}.
  \end{align*}
  Let $n\to\infty$ in this equality. Then since $\|f_n\|_{\mathcal{X}}=1$ for all $n\in\mathbb{N}$, $u_n\to u_\infty$ strongly in $\mathcal{X}$, $\mu_n\to\mu_\infty$, and $|\alpha_n|\to\infty$ as $n\to\infty$, we have $i\mu_\infty(u_\infty,v)_{\mathcal{X}}+i(u_\infty,\Lambda^\ast\mathbb{Q}v)_{\mathcal{X}}=0$. By this equality and $\mathbb{Q}v=v-\mathbb{P}v$ we obtain
  \begin{align} \label{Pf_PC:uin_Last}
    (u_\infty,\Lambda^\ast v)_{\mathcal{X}} = -\mu_\infty(u_\infty,v)_{\mathcal{X}}+(u_\infty,\Lambda^\ast\mathbb{P}v)_{\mathcal{X}}
  \end{align}
  for all $v\in D_{\mathcal{X}}(A)$. Moreover, we see that $\Lambda^\ast\mathbb{P}$ is a closed operator on $\mathcal{X}$ with domain $D_{\mathcal{X}}(\Lambda^\ast\mathbb{P})=\mathcal{X}$, since $D_{\mathcal{X}}(\mathbb{P})=\mathcal{X}$ and $R_{\mathcal{X}}(\mathbb{P})=N_{\mathcal{X}}(\Lambda)\subset D_{\mathcal{X}}(\Lambda^\ast)$ by the condition (c). Hence $\Lambda^\ast\mathbb{P}$ is bounded on $\mathcal{X}$ by the closed graph theorem. By this fact and the density of $D_{\mathcal{X}}(A)$ in $\mathcal{X}$, we find that \eqref{Pf_PC:uin_Last} holds for all $v\in D_{\mathcal{X}}(\Lambda^\ast)$. This shows that \eqref{Pf_PC:uin_Lam} is valid since $\Lambda^{\ast\ast}=\Lambda$ in $\mathcal{X}$. Now we observe by \eqref{Pf_PC:uin_Lam} and $\mathbb{P}\Lambda u_\infty\in N_{\mathcal{X}}(\Lambda)$ that $\Lambda u_\infty \in D_{\mathcal{X}}(\Lambda)$ and $\Lambda^2u_\infty = -\mu_\infty\Lambda u_\infty$. Moreover, $\Lambda u_\infty\neq0$ since $u_\infty\in\mathcal{Y}=N_{\mathcal{X}}(\Lambda)^\perp$ and $u_\infty\neq0$ by \eqref{Pf_PC:uin_del}. Thus $-\mu_\infty\in\mathbb{R}$ is an eigenvalue of $\Lambda$, which yields $\mu_\infty=0$ by the condition (e). Hence $\Lambda^2u_\infty=0$, i.e. $\Lambda u_\infty\in N_{\mathcal{X}}(\Lambda)\cap R_{\mathcal{X}}(\Lambda)$. By this fact and the condition (d) we have $\Lambda u_\infty=0$, i.e. $u_\infty\in N_{\mathcal{X}}(\Lambda)$. However, since $u_\infty$ belongs to $\mathcal{Y}=N_{\mathcal{X}}(\Lambda)^\perp$, it follows that $u_\infty=0$, which contradicts \eqref{Pf_PC:uin_del}. Hence \eqref{E:Phi_Conv} is valid.
\end{proof}

\section{Appendix: Derivation of the vorticity and linearized equations} \label{S:DeVo}

In this section we derive the vorticity equation \eqref{E:Vo_Intro} form the Navier--Stokes equations \eqref{E:NS_Intro}. We also linearize \eqref{E:Vo_Intro} around the $n$-jet Kolmogorov type flow \eqref{E:Zna_Intro}. For the sake of simplicity, we assume that functions are sufficiently smooth in this section. We also note that here we only consider real-valued functions.

Let us introduce notations and formulas from differential geometry. For details, we refer to \cite{Tay11_1,Lee13,Lee18}. We use the spherical coordinate system \eqref{E:SpCo_Intro} so that
\begin{align*}
  \partial_\theta\mathbf{x} =
  \begin{pmatrix}
    \cos\theta\cos\varphi \\
    \cos\theta\sin\varphi \\
    -\sin\theta
  \end{pmatrix},
  \quad \partial_\varphi\mathbf{x} =
  \begin{pmatrix}
    -\sin\theta\sin\varphi \\
    \sin\theta\cos\varphi \\
    0
  \end{pmatrix},
  \quad \mathbf{n}_{S^2} =
  \begin{pmatrix}
    \sin\theta\cos\varphi \\
    \sin\theta\sin\varphi \\
    \cos\theta
  \end{pmatrix}.
\end{align*}
Let $\mathfrak{X}(S^2)$ and $\Omega^k(S^2)$, $k=0,1,2$ be the space of vector fields and $k$-forms on $S^2$. By $\langle\cdot,\cdot\rangle$ we denote the inner product on $\mathfrak{X}(S^2)$ induced by the inner product in $\mathbb{R}^3$. We use the same notation $\langle\cdot,\cdot\rangle$ for the inner product on $\Omega^k(S^2)$. For $\mathbf{u}\in\mathfrak{X}(S^2)$ and $\eta\in\Omega^1(S^2)$ we define $\mathbf{u}^\flat\in\Omega^1(S^2)$ and $\eta^\sharp\in\mathfrak{X}(S^2)$ by $\mathbf{u}^\flat(\mathbf{v})=\langle\mathbf{u},\mathbf{v}\rangle$ and $\langle\eta^\sharp,\mathbf{v}\rangle=\eta(\mathbf{v})$ for $\mathbf{v}\in\mathfrak{X}(S^2)$. Let $\nabla$ be the Levi-Civita connection on $S^2$ and $d$ the exterior derivative on $S^2$. Also, let $\mathfrak{L}_{\mathbf{u}}$ be the Lie derivative along $\mathbf{u}\in\mathfrak{X}(S^2)$.
Then
\begin{align} \label{E:CoLi_Fun}
  \nabla_{\mathbf{u}}\psi = \mathfrak{L}_{\mathbf{u}}\psi = d\psi(\mathbf{u}) = \langle\mathbf{u},\nabla \psi\rangle
\end{align}
for a function $\psi\in\Omega^0(S^2)$, where $\nabla \psi=(d\psi)^\sharp$ is the gradient of $\psi$, and
\begin{align} \label{E:CoLi_One}
  (D_{\mathbf{u}}\eta)(\mathbf{v}) = D_{\mathbf{u}}\bigl(\eta(\mathbf{v})\bigr)-\eta(D_{\mathbf{u}}\mathbf{v}), \quad D_{\mathbf{u}} = \nabla_{\mathbf{u}},\mathfrak{L}_{\mathbf{u}}, \quad \mathbf{v}\in\mathfrak{X}(S^2)
\end{align}
for a one-form $\eta\in\Omega^1(S^2)$. Moreover, since
\begin{align} \label{E:LeCi}
  \nabla_{\mathbf{u}}\langle\mathbf{v},\mathbf{w}\rangle = \langle\nabla_{\mathbf{u}}\mathbf{v},\mathbf{w}\rangle+\langle\mathbf{v},\nabla_{\mathbf{u}}\mathbf{w}\rangle, \quad \mathfrak{L}_{\mathbf{u}}\mathbf{v} = \nabla_{\mathbf{u}}\mathbf{v}-\nabla_{\mathbf{v}}\mathbf{u}
\end{align}
for $\mathbf{u},\mathbf{v},\mathbf{w}\in\mathfrak{X}(S^2)$, it follows from \eqref{E:CoLi_Fun}--\eqref{E:LeCi} that
\begin{align} \label{E:CoLi_Ex}
  (\nabla_{\mathbf{u}}\mathbf{v})^\flat = \nabla_{\mathbf{u}}\mathbf{v}^\flat, \quad \nabla_{\mathbf{u}}\mathbf{v}^\flat+\nabla_{\mathbf{v}}\mathbf{u}^\flat = \mathfrak{L}_{\mathbf{u}}\mathbf{v}^\flat+\mathfrak{L}_{\mathbf{v}}\mathbf{u}^\flat-d\langle\mathbf{u},\mathbf{v}\rangle.
\end{align}
Let $d\mathcal{H}^2$ be the volume form of $S^2$ and $\mathrm{div}\,\mathbf{u}$ the divergence of $\mathbf{u}\in\mathfrak{X}(S^2)$. It is known that, for $\mathbf{u}\in\mathfrak{X}(S^2)$ and $\eta\in\Omega^k(S^2)$, $k=0,1$,
\begin{align} \label{E:LiExt}
  d(\mathfrak{L}_{\mathbf{u}}\eta) = \mathfrak{L}_{\mathbf{u}}(d\eta), \quad \mathfrak{L}_{\mathbf{u}}(d\mathcal{H}^2) = (\mathrm{div}\,\mathbf{u})\,d\mathcal{H}^2.
\end{align}
Let $\ast$ be the Hodge star operator. For $\mathbf{u}\in\mathfrak{X}(S^2)$ we define $\mathrm{rot}\,\mathbf{u}=\ast d\mathbf{u}^\flat\in\Omega^0(S^2)$. It is known that $\ast \psi=\psi\,d\mathcal{H}^2$ for $\psi\in\Omega^0(S^2)$ and $\ast^2=(-1)^{k(2-k)}$ on $\Omega^k(S^2)$, $k=0,1,2$. Also, we easily find that $\ast\mathbf{u}^\flat=(\mathbf{n}_{S^2}\times\mathbf{u})^\flat$ for $\mathbf{u}\in\mathfrak{X}(S^2)$ by using the spherical coordinate system \eqref{E:SpCo_Intro}. Let $d^\ast$ be the formal adjoint of $d$ and $\Delta_H=-(dd^\ast+d^\ast d)$ the Hodge Laplacian on $\Omega^1(S^2)$. By abuse of notation, we write $\Delta_H\mathbf{u}$ for $(\Delta_H\mathbf{u}^\flat)^\sharp$ when $\mathbf{u}\in\mathfrak{X}(S^2)$. Since $S^2$ is a 2D manifold, we have $d^\ast=-\ast d\ast$. Also, for $\mathbf{u}\in\mathfrak{X}(S^2)$ and $\psi\in\Omega^0(S^2)$,
\begin{align} \label{E:dast_Lap}
  d^\ast\mathbf{u} = -\mathrm{div}\,\mathbf{u}, \quad d^\ast d\psi = -\mathrm{div}(\nabla\psi) = -\Delta\psi,
\end{align}
where $\Delta$ is the Laplace--Beltrami operator on $S^2$.

Now let us derive \eqref{E:Vo_Intro}. Let $\mathbf{u}\in\mathfrak{X}(S^2)$ and $p\in\Omega^0(S^2)$ satisfy \eqref{E:NS_Intro} with external force $\mathbf{f}\in\mathfrak{X}(S^2)$. Using \eqref{E:CoLi_Ex} and \eqref{E:dast_Lap}, we rewrite \eqref{E:NS_Intro} as
\begin{align} \label{E:NS_OneF}
  \partial_t\mathbf{u}^\flat+\mathfrak{L}_{\mathbf{u}}\mathbf{u}^\flat-\nu(\Delta_H\mathbf{u}^\flat+2\mathbf{u}^\flat)+d\left(-\frac{|\mathbf{u}|^2}{2}+p\right) = \mathbf{f}^\flat, \quad d^\ast\mathbf{u}^\flat = 0.
\end{align}
Let $\omega=\mathrm{rot}\,\mathbf{u}=\ast d\mathbf{u}^\flat$. We apply $\ast d$ to \eqref{E:NS_OneF}. Then since $d\mathbf{u}^\flat=\omega\,d\mathcal{H}^2$,
\begin{align*}
  \ast d\mathfrak{L}_{\mathbf{u}}\mathbf{u}^\flat = \ast\mathfrak{L}_{\mathbf{u}}(d\mathbf{u}^\flat) = \ast\{(\mathfrak{L}_{\mathbf{u}}\omega)\,d\mathcal{H}^2+\omega\mathfrak{L}_{\mathbf{u}}(d\mathcal{H}^2)\} = \ast\{(\nabla_{\mathbf{u}}\omega)\,d\mathcal{H}^2\} = \nabla_{\mathbf{u}}\omega
\end{align*}
by \eqref{E:CoLi_Fun}, \eqref{E:LiExt}, and $\mathrm{div}\,\mathbf{u}=0$. Also, by $d^2=0$, $d^\ast=-\ast d\ast$, and \eqref{E:dast_Lap},
\begin{align*}
  \ast d\Delta_H\mathbf{u}^\flat = -\ast dd^\ast d\mathbf{u}^\flat= \ast d\ast d\ast d\mathbf{u}^\flat = -d^\ast d\omega = \Delta\omega.
\end{align*}
By these formulas and $d^2=0$ we get the first equation of \eqref{E:Vo_Intro}. It remains to express $\mathbf{u}$ by $\omega$. By $\ast^2=1$ on $\Omega^2(S^2)$, $d^\ast=-\ast d\ast$, and $d^\ast\mathbf{u}^\flat=0$, we have $d\ast\mathbf{u}^\flat=0$. Hence $\ast\mathbf{u}^\flat$ is a closed one-form and, since $S^2$ is simply connected, it is exact, i.e. there exists $\psi\in\Omega^0(S^2)$ such that $\ast\mathbf{u}^\flat=-d\psi$. We may assume $\int_{S^2}\psi\,d\mathcal{H}^2=0$ by subtracting a constant. Then since $\mathbf{u}^\flat=\ast d\psi$ and $\omega=\ast d\mathbf{u}^\flat=\Delta\psi$ by $\ast^2=-1$ on $\Omega^1(S^2)$, $d^\ast=-\ast d\ast$, and \eqref{E:dast_Lap}, and since $\int_{S^2}\omega\,d\mathcal{H}^2=\int_{S^2}d\mathbf{u}^\flat=0$ by the Stokes theorem, we have
\begin{align*}
  \omega = \Delta^{-1}\psi, \quad \mathbf{u}^\flat = \ast d\psi = \ast d\Delta^{-1}\omega = \ast(\nabla\Delta^{-1}\omega)^\flat = (\mathbf{n}_{S^2}\times\nabla\Delta^{-1}\omega)^\flat.
\end{align*}
Hence we obtain \eqref{E:Vo_Intro} from \eqref{E:NS_Intro}. Conversely, if $\omega\in\Omega^0(S^2)$ with zero mean satisfies \eqref{E:Vo_Intro}, then $\mathbf{u}^\flat=\ast d\Delta^{-1}\omega$ satisfies $d^\ast\mathbf{u}^\flat=0$ and $d(\partial_t\mathbf{u}^\flat+\mathfrak{L}_{\mathbf{u}}\mathbf{u}^\flat-\nu(\Delta_H\mathbf{u}^\flat+2\mathbf{u}^\flat)-\mathbf{f}^\flat)=0$ by the above formulas. Then since $S^2$ is simply connected, there exists $q\in\Omega^0(S^2)$ such that $\mathbf{u}^\flat$ and $q$ satisfy \eqref{E:NS_OneF} with $-|\mathbf{u}|^2/2+p$ replaced by $q$. Hence, setting $p=q+|\mathbf{u}|^2/2$, we find by \eqref{E:CoLi_Ex} and \eqref{E:dast_Lap} that $\mathbf{u}$ and $p$ satisfy \eqref{E:NS_Intro}.

Next we consider the function $\omega_n^a$ of the form \eqref{E:Zna_Intro}. Since
\begin{align*}
  \Delta\omega_n^a = -\lambda_n\omega_n^a, \quad \nabla\omega_n^a = a\frac{dY_n^0}{d\theta}\partial_\theta\mathbf{x}, \quad \mathbf{n}_{S^2}\times\partial_\theta\mathbf{x} = \frac{1}{\sin\theta}\partial_\varphi\mathbf{x},
\end{align*}
the corresponding velocity field is of the form
\begin{align} \label{E:una_form}
  \mathbf{u}_n^a = \mathbf{n}_{S^2}\times\nabla\Delta^{-1}\omega_n^a = -\frac{1}{\lambda_n}\mathbf{n}_{S^2}\times\nabla\omega_n^a = -\frac{a}{\lambda_n\sin\theta}\frac{dY_n^a}{d\theta}\partial_\varphi\mathbf{x}.
\end{align}
Hence $\nabla_{\mathbf{u}_n^a}\omega_n^a=-\lambda_n^{-1}\langle\mathbf{n}_{S^2}\times\nabla\omega_n^a,\nabla\omega_n^a\rangle=0$ and $\omega_n^a$ is a stationary solution of \eqref{E:Vo_Intro} with external force $\mathrm{rot}\,\mathbf{f}_n^a=\nu(\lambda_n-2)Y_n^0$. Let us linearize \eqref{E:Vo_Intro} around $\omega_n^a$. We substitute $\omega=\omega_n^a+\tilde{\omega}_n$ for \eqref{E:Vo_Intro} and omit the nonlinear term with respect to $\tilde{\omega}_n$ to get
\begin{align*}
  \partial_t\tilde{\omega}_n = \nu(\Delta\tilde{\omega}_n+2\tilde{\omega}_n)-\nabla_{\mathbf{u}_n^a}\tilde{\omega}_n-\nabla_{\tilde{\mathbf{u}}_n}\omega_n^a, \quad \tilde{\mathbf{u}}_n = \mathbf{n}_{S^2}\times\nabla\Delta^{-1}\tilde{\omega}_n.
\end{align*}
Here the last two terms are of the form
\begin{align*}
  \nabla_{\mathbf{u}_n^a}\tilde{\omega} &= -\frac{a}{\lambda_n\sin\theta}\frac{dY_n^0}{d\theta}\partial_\varphi\tilde{\omega}_n, \\
  \nabla_{\tilde{\mathbf{u}}_n}\omega_n^a &= -\langle\nabla\Delta^{-1}\tilde{\omega},\mathbf{n}_{S^2}\times\nabla\omega_n^a\rangle = -\frac{a}{\sin\theta}\frac{dY_n^a}{d\theta}\partial_\varphi\Delta^{-1}\tilde{\omega}_n
\end{align*}
by \eqref{E:CoLi_Fun}, \eqref{E:una_form}, $\langle\partial_\varphi\mathbf{x},\nabla f\rangle=\partial_\varphi f$ for $f\in\Omega^0(S^2)$, and $\langle\mathbf{n}_{S^2}\times\mathbf{u},\mathbf{v}\rangle=-\langle\mathbf{u},\mathbf{n}_{S^2}\times\mathbf{v}\rangle$ for $\mathbf{u},\mathbf{v}\in\mathfrak{X}(S^2)$. Moreover, by \eqref{E:SpHa} with $m=0$,
\begin{align*}
  \frac{dY_n^0}{d\theta}(\theta) = -\sqrt{\frac{2n+1}{4\pi}}\,\sin\theta\,P_n'(\cos\theta), \quad P_n'(s) = \frac{1}{2^nn!}\frac{d^{n+1}}{ds^{n+1}}(s^2-1)^n.
\end{align*}
Hence the linearized equation for \eqref{E:Vo_Intro} around $\omega_n^a$ is
\begin{align*}
  \partial_t\tilde{\omega}_n = \nu(\Delta\tilde{\omega}_n+2\tilde{\omega}_n)-\frac{a}{\lambda_n}\sqrt{\frac{2n+1}{4\pi}}\,P_n'(\cos\theta)\partial_\varphi(I+\lambda_n\Delta^{-1})\tilde{\omega}_n
\end{align*}
and we get \eqref{E:Li1_Eq} and \eqref{E:Li2_til} when $n=1,2$ by $P_1'(s)=1$, $P_2'(s)=3s$, and $\lambda_n=n(n+1)$.

\section*{Acknowledgments}

The author would like to thank Professor Yasunori Maekawa for fruitful discussions and valuable comments on this work. The work of the author was supported by Grant-in-Aid for JSPS Fellows No. 19J00693.

\bibliographystyle{abbrv}
\bibliography{NSK_LS_Ref}

\end{document}